\documentclass[reqno]{article}
\usepackage{authblk}
\pdfoutput=1
\usepackage{amsmath}
\usepackage{amsfonts,amsthm,amssymb}
\usepackage{mathabx}
\usepackage{bm}
\usepackage{amscd}
\usepackage[latin2]{inputenc}
\usepackage{t1enc}
\usepackage[mathscr]{eucal}
\usepackage{indentfirst}
\usepackage{graphicx}
\usepackage{graphics}
\usepackage{pict2e}
\usepackage{epic}
\numberwithin{equation}{section}
\usepackage[margin=3cm]{geometry}
\usepackage{epstopdf} 
\usepackage[colorlinks,linkcolor=blue]{hyperref}
\usepackage[capitalise,noabbrev]{cleveref}
\usepackage{todonotes}
\usepackage{verbatim}
\crefformat{equation}{(#2#1#3)}
\crefmultiformat{equation}{(#2#1#3)}{ and~(#2#1#3)}{, (#2#1#3)}{ and~(#2#1#3)}
\crefrangeformat{equation}{(#3#1#4) to~(#5#2#6)}
\usepackage[normalem]{ulem}

\allowdisplaybreaks

\theoremstyle{plain}
\newtheorem{Thm}{Theorem}[section]
\newtheorem*{Thm*}{Theorem}
\newtheorem{Lem}[Thm]{Lemma}

\newtheorem{Prop}[Thm]{Proposition}

\theoremstyle{definition}

\newtheorem{Rem}[Thm]{Remark}
\newtheorem{?}[Thm]{Problem}

\newcommand{\p}{\partial}
\newcommand{\R}{\mathbb{R}}
\newcommand{\e}{\varepsilon}

\newcommand{\px}{\partial_x}

\newcommand{\phib}{\mathring{\phi} }
\newcommand{\psib}{\mathring{\psi} }

\newcommand{\red}{\textcolor{red}}
\newcommand{\blue}{\textcolor{blue}}

\newcommand{\abs}[1]{\left\lvert#1\right\rvert}

\newcommand{\norm}[1]{\left\lVert#1\right\rVert}

\begin{document}
	
	\begin{titlepage}
		\title{Optimal decay rates to the contact wave for 1-D compressible Navier-Stokes equations}
		\author{Lingjun Liu$^{1}$}
			%			\thanks{The research is supported by  }
		\author{Shu Wang$^{1}$}
			%			\thanks{The research is supported by  }}
		\author{Lingda Xu$^{2}$}
			%\thanks{ Corresponding author.
			% \\ E-mail addresses: lingjunliu@bjut.edu.cn (L. Liu), wangshu@bjut.edu.cn (S. Wang), \\ 
			%xulingda@tsinghua.edu.cn (L. Xu).}}
			%		\thanks{The research is supported by  }
		
			%			\thanks{The research is supported by  }
			
		\affil{\begin{center}%{flushleft}
				\footnotesize%\qquad\quad
				 $ ^1 $ %College of Mathematics, Faculty of Science, Beijing University of Technology
				 School of Mathematics, Statistics and Mechanics, Beijing University of Technology, Beijing 100124, %P.R.
				China.\\ 
				%\qquad\quad \quad
				 E-mail: lingjunliu@bjut.edu.cn (L. Liu), wangshu@bjut.edu.cn (S. Wang)\\
	%E-mail
	%Academy of Mathematics and Systems Science, Chinese Academy of Sciences, Beijing 100190, China.\\
				%\vspace{0.07cm}
				%\qquad \quad$^2  $ School of Mathematical Sciences, University of Chinese Academy of Sciences, Beijing 100049, China. \\
				\vspace{0.07cm}
				%\qquad\quad
					 $ ^2 $ Department of Mathematics, Yau Mathematical Sciences Center, Tsinghua University, Beijing 100084, China.  \\
				%\qquad\quad\quad
				 E-mail: xulingda@tsinghua.edu.cn (L. Xu)
				\vspace{0.07cm}
				%\qquad\quad$ ^3 $ Yanqi Lake Beijing Institute of Mathematical Sciences and Applications, Beijing 101408, China. %(xulingda@tsinghua.edu.cn).
		\end{center}%{flushleft}
		}

		%\author{$^{1,2}$
			%\thanks{ Corresponding author.\\ Emails:\ fhuang@amt.ac.cn (F. Huang), lirui202@mails.ucas.ac.cn (R. Li), xulingda@tsinghua.edu.cn (L. Xu)}
			%		\thanks{The research is supported by  }
		%}
		
		%\author{Lingda Xu$^{3,4}$
			%			\thanks{The research is supported by  }
		%}
			
		%\affil{\begin{flushleft}
				%\footnotesize\qquad\quad $ ^1 $ Academy of Mathematics and Systems Science, Chinese Academy of Sciences, Beijing 100190, China.\\
				%\vspace{0.07cm}
				%\qquad \quad$^2  $ School of Mathematical Sciences, University of Chinese Academy of Sciences, Beijing 100049, China. \\
				%\vspace{0.07cm}
				%\qquad\quad	 $ ^3 $ Department of Mathematics, Yau Mathematical Sciences Center, Tsinghua University, Beijing 100084, China.  \\
				%\vspace{0.07cm}
				%\qquad\quad$ ^4 $ Yanqi Lake Beijing Institute of Mathematical Sciences and Applications, Beijing 101408, China.
		%\end{flushleft}}
		\date{}

	\end{titlepage}
	
	\maketitle
\begin{abstract}
This paper investigates the decay rates of the contact wave in one-dimensional Navier-Stokes equations. We study two cases of perturbations, with and without zero mass condition, i.e., %which means
 the integration of initial perturbations is zero and non-zero, respectively. For the case without zero mass condition, we obtain the optimal decay rate $(1+t)^{-\frac{1}{2}}$ for the perturbation in $L^\infty$ norm, which provides a positive answer to the conjecture in \cite{HMX}. We applied the anti-derivative method, introducing the diffusion wave to carry the initial excess mass, diagonalizing the integrated system, and estimating the energy of perturbation in the diagonalized system. Precisely, due to the presence of diffusion waves, the decay rates for errors of perturbed system are too poor to get the optimal decay rate. %In this paper,
  We find the dissipation structural in the diagonalized system, see \cref{ds}. This observation makes us able to fully utilize the fact that the sign of the derivative of the contact wave is invariant and %use it 
  to control the terms with poor decay rates in energy estimates. %\red{There are also terms with poor decay rates for the case with zero mass condition.}
{For the case with zero mass condition, there are also terms with poor decay rates.} In this case, %we %use the fact find 
note that there is a cancellation in the linearly degenerate field so that the terms with poor decay rates will not appear in the second equation of the diagonalized system. Thanks to this cancellation and a Poincar\'e type of estimate obtained by a critical inequality introduced by \cite{HLM}, we get the decay rate of $\ln^{\frac{1}{2}} (2+t)$ for $L^2$ norm of anti-derivatives of perturbation and $(1+t)^{-\frac{1}{2}}\ln^{\frac{1}{2}}(2+t)$ for the $L^2$ norm of perturbation itself, the decay rates are %is also
 optimal, which is consistent with the results obtained by using pointwise estimate in \cite{XZ} for the system with artificial viscosity. 
%This paper investigates the decay rate of contact wave in one-dimensional Navier-Stokes equations. We applied the anti-derivative method, first introducing the diffusion wave to carry the initial excess mass, secondly diagonalizing the integrated system, and thirdly estimating the energy of perturbation in the diagonalized system. The main novelty of this article is that it fully utilizes the fact that the sign of the derivative of the contact wave is invariant, and uses it to control the terms with poor decay rate in energy estimates, thus obtaining the decay rate of $(1+t)^{-\frac{1}{2}}$ for $L^{\infty}$ norm of perturbation \red{under non-zero mass condition and $(1+t)^{-\frac{3}{4}}ln^{\frac{1}{2}}(2+t)$ under zero mass condition, respectively}. Due to the presence of diffusion waves, this rate is optimal, providing a positive answer to the conjecture in \cite{HMX}. \red{When no dissipative waves appear, the decay rate is also optimal, which is consistent with the results obtained by using point-wise estimation in \cite{XZ}.} %Xin-Zeng's paper. %The decay rate also is optimal and 
%It is noted that due to the presence of diffusion waves, this rate is optimal, and provides a positive answer to the conjecture in \cite{HMX}.

\end{abstract}
\

%multi-dimensional scalar conservation law, planar shock wave, Cauchy problem, decay rate} \\
\

	% 35L65, %\blue{35Q30}
	 %35L67, %76L05, 
	%35K15
	
  \maketitle
\section{Introduction}
We study the 1-D compressible Navier-Stokes system by
Lagrangian coordinate, which reads,
\begin{align}\label{1.1}
\begin{cases}
v_t-u_x=0, &\\
u_t+p_x=\mu(\frac{u_x}{v})_x, &t>0,\ \ x\in\mathbb R,\\
(e+\frac{u^2}{2})_t+(pu)_x=(\kappa\frac{\theta_x}{v}+\mu\frac{uu_x}{v})_x, &
\end{cases}
\end{align}
where$\ v(x, t)>0, u(x, t), e(x, t)>0, \theta(x, t)>0\ $and$\ p(x, t)$ denote the specific volume, fluid velocity, internal energy, absolute temperature, and pressure, respectively. The positive constants $\mu$ and $\kappa$ are respectively
the viscosity and heat conduction coefficients. The state equations are given as,
\begin{align}\label{assum}
p=R\frac{\theta}{v}=Av^{-\gamma}e^{\frac{\gamma-1}{R}s},\quad e=\frac{R\theta}{\gamma-1}+\text{const},
\end{align}
where $s$ is the entropy, $A>0, R>0$ are constants and $\gamma>1$ is the adiabatic exponent. If $\mu=\kappa=0$, \cref{1.1} becomes a well-known hyperbolic system of conservation laws named compressible Euler system 
\begin{align}\label{Eu}
	\begin{cases}
		v_t-u_x=0, &\\
		u_t+p_x=0, &t>0,\ \ \ x\in\mathbb R,\\
		(e+\frac{u^2}{2})_t+(pu)_x=0, &
	\end{cases}
\end{align}
which may develop singularities no matter how smooth and small the initial data is. Thus, studying the basic wave patterns, shock wave, rarefaction wave and contact discontinuity is important to understand both the local and large-time behavior of solutions for the Euler system. In this paper, since we investigate system \cref{1.1} with viscosity effect, the viscous version of these wave patterns should be considered correspondingly. 

There are fruitful results in this field, for viscous shock wave, we refer to \cite{G,MN} for stability with initial zero mass condition, \cite{HM,SX} for stability without this condition, \cite{HH,MW,VY} for shock with large amplitude, \cite{HRZ,MZ,ZH} for the using of spectrum analysis and \cite{Liu3,LZ3,LZ5} for pointwise estimates of shock wave, while for rarefaction wave, we refer to \cite{MN2,MN3,NYZ}. For the case of periodic perturbations, we refer to \cite{HXY2022}. {Recently, the time-decay rate toward the viscous shock wave for scalar viscous conservation law is obtained in \cite{HuangXu} by an $L^p$ estimate and the area inequality.} There are still many interesting results not been listed here, since this paper focuses on the decay rate of the contact wave.

The contact discontinuity is generated in the linear degenerate characteristic field and is somehow more subtle. Actually, \cite{X} pointed out that the contact discontinuity can not be an asymptotic attractor for \cref{1.1} and \cite{LX} constructed a viscous profile for the system with artificial viscosity, named contact wave, which approaches to contact discontinuity as the viscosity vanishes in $L^p$ norm $p\geq1$ on any finite-time interval. Moreover, \cite{LX} obtained the pointwise estimates for contact wave and this result was extended by \cite{XZ} to the more general systems with artificial viscosity. This kind of stability was named meta-stability. {For the %case of
 Navier-Stokes equations}%\red{Meanwhile}
, \cite{HMS} proved the stability of contact wave for a free boundary problem of \cref{1.1} by applying a Poincar\'e%$\acute{e}$%$é$
-type inequality. Later, for the case with zero mass condition, by a key observation that although the energy of anti-derivatives of perturbation is growth at the rate of $(1+t)^{\frac{1}{4}}$, it can be compensated by the decay in the energy norm of the perturbation itself %derivatives of %the deviation 
which is of the order of %use the decay rate of
 $(1+t)^{-\frac{1}{4}}$, then in \cite{HMX} %can close
  the a priori assumption was closed and the stability result was proved. This result was extended by \cite{HXY} to the general perturbation without zero mass condition. %for $\norm{\phib,\psib,\mathring{\zeta}}_{L^2}$ to close the a priori assumption,
%  \cite{HMX} not only proved the stability result, but also obtained the decay rate of $(1+t)^{-\frac{1}{4}}$ for $\norm{\phib,\psib,\mathring{\zeta}}_{L^\infty}$. 
 % This result was extended by \cite{HXY}. %\red{It is worth mentioning the result of \cite{HLM}, which proved the stability of the contact wave without applying the anti-derivative technique and obtained the stability of a composite wave of rarefaction waves and a contact wave.} 
It is %\blue{worthy to mention}\red
{worth mentioning} the result of \cite{HLM}, which proved the stability of the contact wave without applying the anti-derivative technique and %\blue{also}
 obtained the stability of a composite wave of rarefaction waves and a contact wave. The key of proof in \cite{HLM} is the introduction of an inequality for heat kernel. 

In this paper, we are concerned with the optimal decay rates. %of the perturbation. 
To be more clear, \cite{HMX} and \cite{HXY} studied the decay rate with and without zero mass condition, both of their decay rates are $(1+t)^{-\frac{1}{4}}$ for $L^\infty$ norm. %They conjectured that the optimal decay rate is $(1+t)^{-\frac{1}{2}}$. %$\norm{\phib,\psib,\mathring{\zeta}}_{L^\infty}$. 
In fact, %\blue{Although}
 \cite{LX,XZ} studied the system with artificial viscosity {by using pointwise estimate}, from their results, it can be expected that the optimal decay rate {of the perturbation without zero mass condition}
  is $(1+t)^{-\frac{1}{2}}$ for Navier-Stokes equations;  for the %case of 
  perturbation with zero mass condition, the optimal decay rate is %we get the decay rate of 
 $\ln^{\frac{1}{2}} (2+t)$ for $L^2$ norm of anti-derivatives of perturbation and $(1+t)^{-\frac{1}{2}}\ln^{\frac{1}{2}}(2+t)$ for the $L^2$ norm of perturbation itself. %the decay rates are %is also
% optimal, which is consistent with the results obtained by using pointwise estimate in \cite{XZ} for the system with artificial viscosity. %\red{, and that of the non-zero mass case is $(1+t)^{-\frac{3}{4}}$}
 This conjecture was proposed by \cite{HMX} and remains unsolved. Note that this problem is non-trivial since the contact wave is only an asymptotic solution for Navier-Stokes equations, while ansatz in \cite{LX,XZ} is the exact solution to the systems they studied. In fact, there are attempts to improve the decay rate in \cite{HMX} with the zero mass condition. For example, \cite{HWW} obtained the decay rate of $(1+t)^{-\frac{5}{8}+\nu}$, where $\nu>0$ is a sufficiantly small constant. Later, \cite{Yang1} applied an estimate on heat kernel introduced by \cite{HLM} to improve the decay rate to $(1+t)^{-\frac{5}{8}}\ln^{\frac{1}{2}}({2+}t)$. %\blue{It is worth mentioning that \cite{LX,XZ} indicated the optimal decay rate is $(1+t)^{-\frac{3}{4}}$ in the case that initial perturbations \blue{have}\red{with} zero mass.} 
 \cite{Yang2} used %\blue{a similar idea}\red
 {an idea similar to \cite{HWW}'s} to study the non-zero mass case and improve the decay rate to $(1+t)^{-\frac{3}{8}}$. {However, the optimal decay rates under both zero mass and non-zero mass conditions on the perturbations have still not been achieved.}%are still open.}

{Thus, new ideas should be applied to obtain the optimal decay rate.} %\blue{Thus, to obtain the optimal decay rate, new ideas should be applied.}
% In this paper, 
 We shall study this problem with an anti-derivative technique. For %\blue{this}
 {the compressible Navier-Stokes system \eqref{1.1} under perturbation without zero mass condition}, we first introduce the diffusion waves to carry the initial excess masses of perturbations as \cite{HXY}. Secondly, we diagonalize the integrated system in order to make full use of the fact that the sign of the spatial derivative of the contact wave is invariant. In the diagonalized system, we successfully achieve better control of terms with poor decay rates. %\red{such as terms involving $R_i$}. \red{That is, by careful observation and direct calculations yields}
%\begin{align}%\label{ds}
%	\begin{aligned}
%	\red{\px^{k+1}B^tA_4\px^{k+1}B%=&\frac{\kappa(\gamma-1)}{R\bar{v}}\left[\sqrt{\frac{\gamma-1}{2\gamma}}(\px^{k+1}b_{1}+\px^{k+1}b_{3})+\frac{1}{\sqrt{\gamma}}\px^{k+1}b_{2}\right]^2+\frac{\mu}{\bar{v}}\left[\frac{1}{\sqrt{2}}(\px^{k+1}b_{3}-\px^{k+1}b_{1})\right]^2\\
%	=O(1)\big(\abs{\px^{k+1}\Psi}^2+\abs{\px^{k+1}W}^2\big).}
%\end{aligned}
%\end{align}
%\red{ Thus as shown in subsection \ref{Ees1}, the poor terms $C\bar\delta\tilde K_i$ in the right hand of energy estimates can be absorbed by the above terms.}
 {Next, a more accurate a priori assumption is needed in the higher-order derivatives estimates.} %\blue{Next, a more accurate a priori assumption is essentially needed in the estimates of higher-order derivatives.} 
 Finally, we make full use of the dissipation structure for the diagonalized system. %, and estimate $\norm{\phi}_{H^2}$ as a supplement.
  Through the above methods, we successfully obtained the decay rate of $(1+t)^{-\frac{1}{2}}$. 

For the case with zero mass condition, we observed that the term with poor decay rate appears only in the equation of conservation of momentum. %, so %its effect can be controlled by the intrinsic dissipation in the nonlinear sound wave families. %, and the viscosity and the heat conductivity in compressible Navier-Stokes equations. 
Through a cancellation in the linearly degenerate field, the terms with poor decay rates will not appear in the second equation of the diagonalized system. Thus, we use the sign invariance of the derivative of the contact wave, again controlling the terms with poor decay rates, and further we use an interesting estimate introduced by \cite{HLM} in the integrated system to obtain a Poincar\'e type inequality, we obtain the decay rate of $L^2$ of the anti-derivative is $\ln^{\frac{1}{2}}(2+t)$. %again controlling the terms with poor decay rates, and further we use an interesting estimate in the anti-derivative to obtain a Poincar\'e type inequality, we obtain the decay rate of $L^2$ is A.%Thus we
%\begin{comment}
%Then as shown in subsection \ref{Ees} , the poor term $R_1$ will be estimated by making use of the control as follows: the standard classical energy method involves the estimates of 
%\begin{align}\label{1.5}
%\int_{\mathbb R}\left|\frac{v_1^{-n}}{\lambda_3}\mathring b_3+
%\frac{v_1^{n}}{\lambda_3}\mathring b_1\right|\left|{R}_1\right|dx\le C\bar\delta \mathring G_0+C\bar\delta(1+t)^{-1}.
%\end{align}
%\end{comment}
%\red{Finally, we get the optimal decay rate of $(1+t)^{-\frac{3}{4}}ln^{\frac{1}{2}}(2+t)$ by making use of a Poincar\'e type of estimate obtained by a critical inequality introduced by \cite{HLM}.}%, we get the decay rate of ln 2 p2 ` tq for L normof anti-derivatives of perturbation and p1 ` tq ́ 12 ln 12 p2 ` tq for the L2 norm of perturbation itself, the decay rate is also optimal, which is consistent with the results obtained by using pointwise XZ estimate in [30] for the system with artificial viscosity.} Moreover, these methods can be expected to play a role in other models and problems.

%\subsection{Construction of contact wave}

We now introduce the viscous version of contact discontinuity. Consider the corresponding Euler equation \eqref{Eu} with the Riemann initial data
\begin{align}\label{Eui}
	(v,u,\theta)(x, 0)=\begin{cases}
		(v_-, 0, \theta_-), \text{if} \ x<0;\\
		(v_+, 0, \theta_+), \text{if} \ x>0,
	\end{cases}
\end{align}
where $v_\pm$ and $\theta_\pm$ are given positive constants. 
It is known that the contact discontinuity solution for \eqref{Eu} with \eqref{Eui}  takes the form:
\begin{align}
	(\bar{V}, \bar{U}, \bar{\Theta})(x, t)=\begin{cases}
		(v_-, 0, \theta_-), x<0, t>0;\\
		(v_+, 0, \theta_+), x>0, t>0,
	\end{cases}
\end{align}
provided that
\begin{align}
	p_-:= \frac{R\theta_-}{v_-}:= p_+:= \frac{R\theta_+}{v_+}.
\end{align}
As \cite{HLM},  we consider the nonlinear diffusion equation
\begin{align}\label{cd1}
	\hat{\Theta}_t=a\left(\frac{\hat{\Theta}_x}{\hat{\Theta}}\right)_x, \ \ \hat{\Theta}({\pm}\infty, t)=\theta_{\pm},\quad a=\frac{\kappa p_+(\gamma-1)}{\gamma R^2}>0.
\end{align}
Due to \cite{HL}, the above equation admits a unique self-similarity solution $\hat{\Theta}(x, t) = \hat{\Theta}(\xi), \xi = \frac{x}{\sqrt{1+t}}$.  Moreover,  $\hat{\Theta}(\xi)$ is a monotone function, increasing if $\theta_+>\theta_-$ and decreasing if $\theta_+<\theta_-$; there exists some positive
constant $\bar{\delta}$, such that for $\delta=|\theta_+-\theta_-|\leq\bar{\delta}, \hat{\Theta}$ satisfies:
\begin{align}\label{cde}
	(1+t)|\hat{\Theta}_{xx}|+(1+t)^{1/2}|\hat{\Theta}_x|+|\hat{\Theta}-\theta_{\pm}|\leq c_1\delta e^{-\frac{c_2x^2}{1+t}},\ \ \ \ as\ |x|\rightarrow\infty,
\end{align}
where $c_1$ and $c_2$ are both positive constants depending only on $\theta_{\pm}$. %Once $\hat{\Theta}$ is determined, 
Then the contact discontinuity profile $(\bar{v}, \bar{u}, \bar{\theta})(x, t)$ can be defined as follows,
\begin{align}\label{c}
	\bar{v}=\frac{R}{p_+}\hat{\Theta},\quad \bar{u}=u_-+\frac{\kappa(\gamma-1)}{\gamma R}\frac{\hat{\Theta}_x}{\hat{\Theta}},\quad \bar{\theta}=\hat{\Theta}-\frac{\gamma-1}{2R}\bar{u}^2, 
\end{align}
and 
\begin{align}\label{Ep}
	\bar{E}=\frac{R}{\gamma-1}\bar{\theta}-\frac{\bar{u}^2}{2},\quad \bar{p}=\frac{R\bar{\theta}}{\bar{v}}.
\end{align}
By a direct calculation, one has %It is straightforward to check that the contact discontinuity $(\bar{v}, \bar{u}, \bar{\theta})$ solves the compressible Navier-Stokes system $(\ref{1.1})$ time asymptotically, that is, 
\begin{align}\label{cd}
	\begin{cases}
		\bar{v}_t-\bar{u}_x=0,\\
		\bar{u}_t+(\frac{R\bar{\theta}}{\bar{v}})_x=\mu(\frac{\bar{u}_x}{\bar{v}})_x+R_{1x},\\
		(\bar{e}+\frac{\bar{u}^2}{2})_t+(\bar{p}\bar{u})_x=(\kappa\frac{\bar{\theta}_x}{\bar{v}}+\mu\frac{\bar{u}\bar{u}_x}{\bar{v}})_x+R_{2x},
	\end{cases}
\end{align}
where
\begin{align}
	R_1=&\left(\frac{\kappa(\gamma-1)}{\gamma R}-\mu\right)\frac{\bar{u}_x}{\bar{v}}+\bar{p}-p_+=O(1)%\delta
	 D_{-1},\label{vr}\\
	R_2=&\left(\frac{\kappa(\gamma-1)}{\gamma R}-\mu\right)\frac{\bar{u}\bar{u}_x}{\bar{v}}+\left(\bar{p}-p_+\right)\bar{u}=O(1)%\delta 
	D_{-\frac{3}{2}},\label{vr2}\\
	D_{-\alpha}=&\delta(1+t)^{-\alpha}\text{exp}\{-\frac{c_2x^2}{1+t}\},\ \  \delta=\max\{|v_+-v_-|,|u_+-u_-|,|\theta_+-\theta_-|\}.
\end{align}
%We first consider the case that the initial excess mass is zero. 
Since the larger error term in $R_1$ is not integrable in time due to its rate of decay of $O(\frac{1}{1+t})$, we need to overcome this main technical difficulty in this paper. % 11+t ).
%this is one of the main technical difficulties we need to overcome.
 %One of key observations in this paper is that 
Because the worst error term $R_1$ appears only in $\eqref{cd}_2$, %the equation of conservation of momentum, 
{its effect can be controlled by the intrinsic dissipation in the nonlinear sound wave families,  and the viscosity and the heat conductivity in the compressible Navier-Stokes system \eqref{1.1}, which will be
%will be 
sufficient for the energy estimates %{sup-norm stability}
 for the contact wave $(\bar v,\bar u,\bar\theta)$ as we present later. }

The rest of this paper is arranged as follows. We will introduce the construction of ansatz and give the main results in section 2. Then we reformulate the system and carry out the energy estimates for the case of non-zero initial excess mass in section 3. In section 4, with the help of critical inequality and a Poincar\'e type of estimate, we show the energy estimates for the case of zero initial excess mass and get the %optimal 
decay rates.

 \textbf{Notations.} 
   We denote $\|u\|_{L^p}$ by the norm of Sobolev space $L^p(\mathbb{R})$, especially $\|\cdot\|_{L^2}:=\|\cdot\|$, $C$ and $c$ %and $\bar{c}$ 
   by the generic positive constants.

\section{Ansatz and the main result}
{In this section, we are ready to give two main theorems about the decay rate of perturbation with and without zero mass conditions, respectively.  We first construct the ansatz for the case of non-zero initial mass. Since some additional diffusion waves in the sound wave families must be generated for general perturbations as observed in \cite{LX,X}, the assumption of non-zero initial mass is generic. }%but complex.}

Denote the conserved quantities by $$
m(x, t)=\left(v, u, \theta+\frac{\gamma-1}{2 R} u^2\right)^t, \quad \bar{m}(x, t)=\left(\bar{v}, \bar{u}, \bar{\theta}+\frac{\gamma-1}{2 R} \bar{u}^2\right)^t,
$$
where ()$^t$ means the transpose of the vector (). %\blue{At the far fields}
{For} $x= \pm \infty$, {the vectors $m$ and $\bar{m}$ are the same, that is} $m_{ \pm}=\left(v_{ \pm}, 0, \theta_{ \pm}\right)^t$. %\blue{Since we consider a general initial perturbation here,} 
The integral $\int_{-\infty}^{\infty}(m(x, 0)-\bar{m}(x, 0)) d x$ may not be zero %\blue{in general}
 {for a general initial perturbation}. %Hence,
Inspired by \cite{LX,X}, we know that the mass distributes in all characteristic fields when time evolves%, which introduces diffusion waves in the two nonlinear sound wave families as in [45] and [29]
. In fact, let
$$
A(v, u, \theta)=\left(\begin{array}{ccc}
	0 & -1 & 0 \\
	-\frac{p}{v} & 0 & \frac{R}{v} \\
	-\frac{(\gamma-1) p u}{R v} & \frac{\gamma-1}{R} p & \frac{(\gamma-1) u}{v}
\end{array}\right),
$$
be the Jacobi matrix of the flux $\left(-u, p, \frac{\gamma-1}{R} p u\right)^t$. It is straightforward to check that the first eigenvalue of $A\left(v_{-}, 0, \theta_{-}\right)$ is $\lambda_1^{-}=-\sqrt{\frac{\gamma p_{-}}{v_{-}}}$ with right eigenvector
$$
r_1^{-}=\left(-1, \lambda_1^{-}, \frac{\gamma-1}{R} p_{-}\right)^t \text {. }
$$
Similarly, the third eigenvalue of $A\left(v_{+}, 0, \theta_{+}\right)$ is $\lambda_3^{+}=\sqrt{\frac{\gamma p_{+}}{v_{+}}}$ with right eigenvector %the third eigenvalue and right eigenvector of $A\left(v_{+}, 0, \theta_{+}\right)$ are respectively $\lambda_3^{+}=\sqrt{\frac{\gamma p_{+}}{v_{+}}}$ and
$$
r_3^{+}=\left(-1, \lambda_3^{+}, \frac{\gamma-1}{R} p_{+}\right)^t.
$$
By strict hyperbolicity, the vectors $r_1^{-}, m_{+}-m_{-}=\left(v_{+}-v_{-}, 0, \theta_{+}-\theta_{-}\right)^t$ and $r_3^{+}$are linearly independent in $\R^3$. Thus, one has %the integral $\int_{-\infty}^{\infty}(m(x, 0)-\bar{m}(x, 0)) d x$ %can be \red{assigned} to %as follows
\begin{align}\label{Z}
\int_{-\infty}^{\infty}(m(x, 0)-\bar{m}(x, 0)) d x=\bar{\theta}_1 r_1^{-}+\bar{\theta}_2\left(m_{+}-m_{-}\right)+\bar{\theta}_3 r_3^{+},
\end{align}
with some constants $\bar{\theta}_i, i=1,2,3.$
\begin{comment}
 \red{ satisfying }%, which are small enough and  controlled by the initial value, i.e.,}
\begin{align}\label{smallconstants}
\red{\left|\bar{\theta}_1\right|+\left|\bar{\theta}_3\right| \leqslant C\varepsilon_0,}%\bar\theta_i\le\e_0.
\end{align}
%$\left|\bar{\theta}_1\right|+\left|\bar{\theta}_3\right| \leqslant C \varepsilon_0$
 \red{for some constant $C>0$.}
 \end{comment}
 %Now
  For all $t \geqslant 0$, motived by \cite{HXY,Liu1}, we define the ansatz $\tilde{m}(x, t)$ by
$$
\tilde{m}(x, t)=\bar{m}\left(x+\bar{\theta}_2, t\right)+\bar{\theta}_1 \theta_1 r_1^{-}+\bar{\theta}_3 \theta_3 r_3^{+},
$$
where
$$
\theta_1(x, t)=\frac{1}{\sqrt{4 \pi(1+t)}} e^{-\frac{\left(x-\lambda_1^{-}(1+t)\right)^2}{4(1+t)}}, \quad \theta_3(x, t)=\frac{1}{\sqrt{4 \pi(1+t)}} e^{-\frac{\left(x-\lambda_3^{+}(1+t)\right)^2}{4(1+t)}},
$$
satisfying
$$
\theta_{1 t}+\lambda_1^{-} \theta_{1 x}=\theta_{1 x x}, \quad \theta_{3 t}+\lambda_3^{+} \theta_{3 x}=\theta_{3 x x},
$$
and $\int_{-\infty}^{\infty} \theta_i(x, t) d x=1$ for $i=1,3$%, and all $t \geqslant 0$ respectively
. More precisely, denote the ansatz $\tilde{m}$ as the following expression
$$
\tilde{m}(x, t)=\left(\tilde{v}, \tilde{u}, \tilde{\theta}+\frac{\gamma-1}{2 R} \tilde{u}^2\right)^t(x, t),
$$
with
\begin{align}\label{ansatz}
\begin{aligned}
	& \tilde{v}(x, t)=\bar{v}\left(x+\bar{\theta}_2, t\right)-\bar{\theta}_1 \theta_1-\bar{\theta}_3 \theta_3, \\
	& \tilde{u}(x, t)=\bar{u}\left(x+\bar{\theta}_2, t\right)+\lambda_1^{-} \bar{\theta}_1 \theta_1+\lambda_3^{+} \bar{\theta}_3 \theta_3, \\
	& \tilde{\theta}(x, t)=\bar{\theta}\left(x+\bar{\theta}_2, t\right)+\frac{\gamma-1}{2 R} \bar{u}^2\left(x+\bar{\theta}_2, t\right)+\frac{\gamma-1}{R} p_{+}\left(\bar{\theta}_1 \theta_1+\bar{\theta}_3 \theta_3\right)-\frac{\gamma-1}{2 R} \tilde{u}^2 .
\end{aligned}
\end{align}
Furthermore, by direct calculation, it holds that
\begin{align}\label{mm1}
\begin{aligned}
	& \int_{-\infty}^{\infty}(m(x, 0)-\tilde{m}(x, 0)) d x \\
	& =\int_{-\infty}^{\infty}(m(x, 0)-\bar{m}(x, 0)) d x+\int_{-\infty}^{\infty}(\bar{m}(x, 0)-\tilde{m}(x, 0)) d x \\
	& =\bar{\theta}_2\left(m_{+}-m_{-}\right)+\int_{-\infty}^{\infty}\left(\bar{m}(x, 0)-\bar{m}\left(x+\bar{\theta}_2, 0\right)\right) d x=0 .
\end{aligned}
\end{align}
Without loss of generality, we can assume that $\bar{\theta}_2=0$ from now on. The system for $\tilde{m}$ is
\begin{align}\label{sys-ansatz}
\left\{\begin{array}{l}
	\tilde{v}_t-\tilde{u}_x=\tilde{R}_{1 x}, \\
	\tilde{u}_t+\tilde{p}_x=\mu\left(\frac{\tilde{u}_x}{\tilde{v}}\right)_x+\tilde{R}_{2 x}, \\
	\left(\tilde{e}+\frac{\tilde{u}^2}{2}\right)_t+(\tilde{p} \tilde{u})_x=\kappa\left(\frac{\tilde{\theta}_x}{\tilde{v}}\right)_x+\left(\frac{\mu}{\tilde{v}} \tilde{u} \tilde{u}_x\right)_x+\tilde{R}_{3 x},
\end{array}\right.
\end{align}
where
\begin{align}\label{tildeR}
\begin{aligned}
	\tilde{R}_1=&-\bar{\theta}_1 \theta_{1 x}-\bar{\theta}_3 \theta_{3 x}, \\
	\tilde{R}_2=&R_1+\mu\left(\frac{\bar{u}_x}{\bar{v}}-\frac{\tilde{u}_x}{\tilde{v}}\right)+\left(\lambda_1^{-} \bar{\theta}_1 \theta_{1 x}+\lambda_3^{+} \bar{\theta}_3 \theta_{3 x}\right)+\left(\tilde{p}-\bar{p}-\lambda_1^{-2} \bar{\theta}_1 \theta_1-\lambda_3^{+2} \bar{\theta}_3 \theta_3\right),\\
	\tilde{R}_3= & R_2+\kappa\left(\frac{\bar{\theta}_x}{\bar{v}}-\frac{\tilde{\theta}_x}{\tilde{v}}\right)+\mu\left(\frac{\bar{u} \bar{u}_x}{\bar{v}}-\frac{\tilde{u} \tilde{u}_x}{\tilde{v}}\right)+p_{+}\left(\bar{\theta}_1 \theta_{1 x}+\bar{\theta}_3 \theta_{3 x}\right) \\
	& +\left(\tilde{p} \tilde{u}-\bar{p} \bar{u}-p_{+} \lambda_1^{-} \bar{\theta}_1 \theta_1-p_{+} \lambda_3^{+} \bar{\theta}_3 \theta_3\right) .
\end{aligned}
\end{align}
By \cref{cde,Ep,ansatz}, it yields
$$
\begin{aligned}
	\tilde{p}-\bar{p}= & -\frac{\bar{p}}{\bar{v}}(\tilde{v}-\bar{v})+\frac{R}{\bar{v}}(\tilde{\theta}-\bar{\theta})+O(1)\left[(\tilde{v}-\bar{v})^2+(\tilde{\theta}-\bar{\theta})^2\right] \\
	= & \frac{\gamma p_{-}}{v_{-}} \bar{\theta}_1 \theta_1+\frac{\gamma p_{+}}{v_{+}} \bar{\theta}_3 \theta_3  +O\left(\delta+\bar{\theta}_1^2+\bar{\theta}_3^2\right) \frac{1}{1+t}\left(e^{-\frac{c x^2}{1+t}}+e^{-\frac{c\left(x-\lambda_1^{-}(1+t)\right)^2}{1+t}}+e^{-\frac{c\left(x-\lambda_3^{+}(1+t)\right)^2}{1+t}}\right),
\end{aligned}
$$
for some positive constant $c>0$. Also, similar estimate holds for $\tilde{p} \tilde{u}-\bar{p} \bar{u}$. Thus there is
\begin{align}\label{eqRi}
\tilde{R}_i=O\left(\delta+\bar{\theta}_1^2+\bar{\theta}_3^2\right) \frac{1}{1+t}\left(e^{-\frac{c x^2}{1+t}}+e^{-\frac{c\left(x-\lambda_1^{-}(1+t)\right)^2}{1+t}}+e^{-\frac{c\left(x-\lambda_3^{+}(1+t)\right)^2}{1+t}}\right),
\end{align}
for $i=1,2,3$.

{Let $(v,u,\theta)$  be the solutions of the Cauchy problem for Navier-Stokes equations \eqref{1.1} with the initial data}
\begin{align}\label{inti0}
(v,u,\theta)(x,t=0)=(v_0,u_0,\theta_0)(x).
\end{align}
Denote the perturbation around the ansatz $(\tilde{v}, \tilde{u}, \tilde{\theta})$ by
\begin{align}\label{ppz}
\phi(x, t)=v-\tilde{v}, \quad \psi(x, t)=u-\tilde{u}, \quad \zeta(x, t)=\theta-\tilde{\theta},
\end{align}
and set
\begin{align}\label{anti}
\begin{gathered}
	\Phi(x, t)=\int_{-\infty}^x \phi(y, t) d y, \quad \Psi(x, t)=\int_{-\infty}^x \psi(y, t) d y, \\
	\bar{W}(x, t)=\int_{-\infty}^x\left(e+\frac{|u|^2}{2}-\tilde{e}-\frac{|\tilde{u}|^2}{2}\right)(y, t) d y .
\end{gathered}
\end{align}
%Notice that the quantities $(\Phi, \Psi, \bar{W})$ can be well defined in some Sobolev space s
Since the compressible Navier-Stokes equations \eqref{1.1} and the system \cref{sys-ansatz} are in the conservative forms and $(\Phi, \Psi, \bar{W})(\pm \infty, 0)=0$ due to \eqref{mm1}, the quantities $(\Phi, \Psi, \bar{W})$ can be well defined in some Sobolev space.

There is the main result for the compressible Navier-Stokes equations \eqref{1.1}-\eqref{assum} under the perturbation with non-zero initial excess mass as follows.
\begin{Thm}%[\red{Non-zero %initialmass}]
\label{mt}
	Let $(\bar{v}, \bar{u}, \bar{\theta})(x, t)$ be the contact wave defined in \cref{c}, then there exists small positive constants $\varepsilon_0$, $\delta_0$, such that if the initial data $\left(v_0, u_0, \theta_0\right)$ satisfies
	\begin{align}
		\delta&:=\left|\theta_{+}-\theta_{-}\right| \leq \delta_0,\qquad
		\varepsilon:=	\|(\Phi, \Psi, \bar{W})(\cdot,0)\|_{L^2}+\|(\phi, \psi, \zeta)(\cdot,0)\|_{H^1}+\norm{\phi_{xx}(\cdot,0)}_{L2} \leq \varepsilon_0,
	\end{align}
 the system \cref{1.1} admits a unique global-in-time solution $(v, u, \theta)(x, t)$ satisfying
\begin{align}
	\begin{aligned}
		(\Phi, \Psi, \bar{W})  \in C\left(0,+\infty ; H^2\right), \quad
		(\phi,\psi, \zeta)  \in L^2\left(0,+\infty ; H^2\right),
	\end{aligned}
\end{align}
	where  $(\Phi, \Psi, \bar{W})$ and $(\phi, \psi, \zeta)$ are the perturbations defined in \cref{ppz,anti}. Moreover, the following decay rate holds,
	\begin{align}\label{L2decay}
	\begin{aligned}
	&\norm{\big(\phi, \psi, \zeta\big)}_{L^{2}} \leq C\left(\varepsilon_0+\delta_0^{\frac{1}{4}}\right)(1+t)^{-\frac{1}{4}},\\
	&\norm{\big(\phi, \psi, \zeta\big)_x}_{L^{2}} \leq C\left(\varepsilon_0+\delta_0^{\frac{1}{4}}\right)(1+t)^{-\frac{3}{4}},
	\end{aligned}
	\end{align}
	and
	\begin{align}\label{decay}
	\begin{aligned}
		\norm{\big(v-\bar{v}, u-\bar{u}, \theta-\bar{\theta}\big)}_{L^{\infty}} \leq C\left(\varepsilon_0+\delta_0^{\frac{1}{4}}\right)(1+t)^{-\frac{1}{2}},
	\end{aligned}
		\end{align}
	where ${C}$ is a positive constant independent of time.
\end{Thm}

\begin{Rem}
\eqref{decay} is optimal due to the presence of diffusion waves, cf. \cite{LX,XZ}, which gives a positive answer to \cite{HMX}. 
%\red{In this paper, we apply the energy estimates to verify the conjecture of Huang-Matsumura-Xin-Yang in \cite{HMX,HXY}, and obtain the optimal decay rate of $(1+t)^{-\frac{1}{2}}$ for the perturbation in $L^\infty$ norm, which has been given in \cite{XZ} by using pointwise estimates.} %. In \cite{XZ}, Xin-Zeng used pointwise estimates to get the same decay rate, and pointed out that is optimal.}% which  %which has been given in \cite{XZ} by using pointwise estimate.}
\end{Rem}
\ 

Now we introduce the case of zero initial excess mass.
The perturbations are given by 
\begin{align}\label{ppzz}
	\phib(x, t)=v-\bar{v}, \quad \psib(x, t)=u-\bar{u}, \quad %\blue{\bar{\zeta}}
	{\mathring\zeta}
	(x, t)=\theta-\bar{\theta},
\end{align}
{and set the anti-derivative variables} 
\begin{align}\label{anti0}
\begin{gathered}
	{\mathring\Phi(x, t)=\int_{-\infty}^x \phib(y, t) d y, \quad \mathring\Psi(x, t)=\int_{-\infty}^x \psib(y, t) d y}, \\
	{\mathring{\bar{W}}(x, t)=\int_{-\infty}^x\left(e+\frac{|u|^2}{2}-\tilde{e}-\frac{|\bar{u}|^2}{2}\right)(y, t) dy.}
\end{gathered}
\end{align}
%\Phib(x,t)=\int_
{%Since
Because of the conservative structure of both the Navier-Stokes system \eqref{1.1} and the approximate system \eqref{cd}, the quantities $(\mathring\Phi,\mathring\Psi,\mathring{\bar W})(\cdot,t)\in L^2(\mathbb{R})$ can be provided that the initial {excess} mass is zero, i.e.,}
\begin{align}\label{masszero1}
(\mathring\Phi,\mathring\Psi,\mathring{\bar W})(+\infty,0)=\int_{-\infty}^{+\infty}(v_0(x)-\bar v(x,0),u_0(x)-\bar u(x,0), E_0(x)-\bar E(x,0))dx=(0,0,0),
\end{align}
where $E=e+\frac{|u|^2}{2}$ is the total energy. Then the main result concerning zero initial {excess} mass is as follows. %the initial excess mass is zero is
\begin{Thm}%[\red{Zero %initial mass}]
\label{mt0}
	Let $(\bar{v}, \bar{u}, \bar{\theta})(x, t)$ be the contact wave defined in \cref{c}, then there exists small positive constants $\varepsilon_0$, $\delta_0$, such that if the initial data $\left(v_0, u_0, \theta_0\right)$ satisfies
	\begin{align}
		\delta&:=\left|\theta_{+}-\theta_{-}\right| \leq \delta_0,\qquad
		\varepsilon:=	\|(\mathring\Phi, \mathring\Psi, \mathring{\bar{W}})(\cdot,0)\|_{L^2}+\|(\phib, \psib, \mathring\zeta)(\cdot,0)\|_{H^1}+\norm{\phib_{xx}(\cdot,0)}_{L2} \leq \varepsilon_0,
	\end{align}
 the system \cref{1.1} admits a unique global-in-time solution $(v, u, \theta)(x, t)$ satisfying
\begin{align}
	\begin{aligned}
		(\mathring\Phi, \mathring\Psi,\mathring{ \bar{W}})  \in C\left(0,+\infty ; H^2\right), \quad
		(\phib,\psib, {\mathring\zeta})  \in L^2\left(0,+\infty ; H^2\right),
	\end{aligned}
\end{align}
	where  $(\mathring\Phi, \mathring\Psi, \mathring{\bar{W}})$ and $(\phib, \psib, {\mathring\zeta})$ are the perturbations defined in \cref{ppzz,anti0}. Moreover, the following decay rate holds,
	\begin{align}\label{L2decay0}
	%\begin{aligned}
	\norm{\big(\phib, \psib, \mathring\zeta\big)}_{L^{2}} \leq C\left(\varepsilon_0+\delta_0^{\frac{1}{4}}\right)(1+t)^{-\frac{1}{2}}\ln^{\frac{1}{2}}(2+t),
	\end{align}
	\begin{align}\label{L2decay1}
	\norm{\big(\phib, \psib, \mathring\zeta\big)_x}_{L^{2}} \leq C\left(\varepsilon_0+\delta_0^{\frac{1}{4}}\right)(1+t)^{-1}\ln^{\frac{1}{2}}(2+t),
	%\end{aligned}
	\end{align}
	and
	\begin{align}\label{decay0}
	\begin{aligned}
		\norm{\big(v-\bar{v}, u-\bar{u}, \theta-\bar{\theta}\big)}_{L^{\infty}} \leq C(\e+\bar\delta^{\frac{1}{4}})(1+t)^{-\frac{3}{4}}\ln^{\frac{1}{2}}(2+t),%\blue{C\left(\varepsilon_0+\delta_0^{\frac{1}{4}}\right)(1+t)^{-\frac{1}{2}},}
	\end{aligned}
		\end{align}
	where ${C}$ is a positive constant independent of time.
\end{Thm}

\begin{Rem}
%\red{In this paper, the decay rate of $(1+t)^{-\frac{3}{4}}ln^{\frac{1}{2}}(2+t)$ for the perturbation in $L^\infty$ norm 
\eqref{L2decay0} is %and \eqref{L2decay1} are
 optimal, and is consistent with the results
\begin{align}
\begin{aligned}
\|u(\cdot,t)-u^a(\cdot,t)\|_{L^p(\mathbb R)}=\begin{cases}
		(1+t)^{-1+\frac{1}{p}}, \ \ \ 1<p<2;\\
		(1+t)^{-\frac{1}{2}}\ln^{\frac{1}{2}}(1+t), \ \ \ p=2;\\
		(1+t)^{-\frac{3}{4}+\frac{1}{2p}}, \ \ \ 2<p\le\infty,
	\end{cases}
\end{aligned}
\end{align}
that obtained by using pointwise estimate for vicous conservation laws in \cite{XZ}.%It is the optimal decay rate in this paper}
\end{Rem}

\begin{Rem}
{In the case of non-zero initial excess mass, the coupling of diffusion waves (as $\theta_1(x,t)$ and $\theta_3(x,t)$) and the viscous contact wave leads to some poor decay terms such as $\tilde R_i$ in \eqref{tildeR}, so that the usual energy approach cannot be applied. For the case of zero initial excess mass, compared with general perturbation, it is easier to achieve the optimal decay rate due to the additional diffusion waves.}% than that for general perturbation.}%It is the optimal decay rate in this paper}
\end{Rem}

\section{%\blue{Energy estimates for the case of}
 Non-zero initial mass}
 {In this section, we aim to prove the main theorem \ref{mt}. Firstly, we reformulated the compressible Navier-Stokes equations to an integrated system. Then we give a diagonalized system to control terms with poor decay rates by exploiting an intrinsic dissipation associated with the contact wave. Finally, the energy estimation is used to close a priori assumptions, and then the optimal decay rate of the perturbation in $L^\infty $ norm is obtained by utilizing Gronwall's inequality.}
\subsection{Reformulated system}
{In order to take advantage of the structure of the underlying contact wave, use the intrinsic dissipation in the compressible Navier-Stokes system and overcome the strong nonlinearity, we introduce the anti-derivative variables $(\Phi,\Psi,\bar W)$ in \eqref{anti}.}
%Firstly
Thus, direct calculations yield the following relationship,
$$
(\phi, \psi)=(\Phi, \Psi)_x\quad \text{and}\quad \frac{R}{\gamma-1} \zeta+\frac{1}{2}\left|\Psi_x\right|^2+\tilde{u} \Psi_x=\bar{W}_x.
$$
Subtracting \eqref{sys-ansatz} from equation \eqref{1.1} and integrating the resulting system yield%Subtracting \eqref{cd} from equation \eqref{1.1} and integrating the resulting system yield %Then we study
 the following integrated system,
\begin{align}\label{sys-Ori}
\left\{\begin{array}{l}
	\Phi_t-\Psi_x=-\tilde{R}_1, \\
	\Psi_t+p-\tilde{p}=\frac{\mu}{v} u_x-\frac{\mu}{\tilde{v}} \tilde{u}_x-\tilde{R}_2, \\
	\bar{W}_t+p u-\tilde{p} \tilde{u}=\frac{\kappa}{v} \theta_x-\frac{\kappa}{\tilde{v}} \tilde{\theta}_x+\frac{\mu}{v} u u_x-\frac{\mu}{\tilde{v}} \tilde{u} \tilde{u}_x-\tilde{R}_3.
\end{array}\right.
\end{align}
%which is given by \cref{1.1,sys-ansatz}.
 To capture the diffusive effect, we introduce another variable for convenience,
\begin{align}\label{WZEta}
W=\frac{\gamma-1}{R}(\bar{W}-\tilde{u} \Psi), \qquad
\Rightarrow\qquad 
\zeta=W_x-Y, \quad \text { with\quad } Y=\frac{\gamma-1}{R}\left(\frac{1}{2} \Psi_x^2-\tilde{u}_x \Psi\right) .
\end{align}
Using the new variable $W$ and linearizing the left-hand side of the system \eqref{sys-Ori}%(3.1)
, we have
\begin{align}\label{sys-pert}
\left\{\begin{array}{l}
	\Phi_t-\Psi_x=-\tilde{R}_1, \\
	\Psi_t-\frac{p_{+}}{\tilde{v}} \Phi_x+\frac{R}{\tilde{v}} W_x=\frac{\mu}{\tilde{v}} \Psi_{x x}+Q_1, \\
	\frac{R}{\gamma-1} W_t+p_{+} \Psi_x=\frac{\kappa}{\tilde{v}} W_{x x}+Q_2,
\end{array}\right.
\end{align}
where
\begin{align}
&\begin{aligned}\label{J}
	J_1=&\frac{\tilde{p}-p_{+}}{\tilde{v}} \Phi_x-\left[p-\tilde{p}+\frac{\tilde{p}}{\tilde{v}} \Phi_x-\frac{R}{\tilde{v}}(\theta-\tilde{\theta})\right]=O(1)\left(\Phi_x^2+W_x^2+Y^2+|\tilde{u}|^4{+\theta_1^2+\theta_3^2}\right), \\
	J_2=&\left(p_{+}-p\right) \Psi_x=O(1)\left(\Phi_x^2+\Psi_x^2+W_x^2+Y^2+|\tilde{u}|^4{+\theta_1^2+\theta_3^2}\right), 
\end{aligned}\\
&\begin{aligned}\label{Q}
	Q_1=&\left(\frac{\mu}{v}-\frac{\mu}{\tilde{v}}\right) u_x+J_1+\frac{R}{\tilde{v}} Y-\tilde{R}_2, \\
	Q_2=&\left(\frac{\kappa}{v}-\frac{\kappa}{\tilde{v}}\right) \theta_x+\frac{\mu u_x}{v} \Psi_x-\tilde{R}_3-\tilde{u}_t \Psi+\tilde{u} \tilde{R}_2+J_2-\frac{\kappa}{\tilde{v}} Y_x .
\end{aligned}
\end{align}
Also, one can obtain the estimates for derivatives
\begin{align}\label{Jx}
	\begin{aligned}
	J_{1x}=&O(\delta) D_{-\frac{1}{2}} \left(\Phi_x^2+W_x^2+Y^2+{|\bar u_x|^2+|\bar u|^4+\theta_1^2+\theta_3^2}%+|\tilde{u}|^4
	\right)\\
	&+{O(1)\left[\left(|\bar u|^2+\theta_1+\theta_3\right)\Phi_{xx}+\Phi_{xx}^2+W_{xx}^2+Y_x^2\right]},\\
	%\Phi_{xx}+W_{xx}+Y_x%+|\tilde u\tilde u_x|
	%+|\tilde u|^4+\theta_1^2+\theta_3^2\right)},
	%&\blue{+O(1)\left(\Phi_x\Phi_{xx}+W_xW_{xx}+YY_x\right)},\\
	J_{2x}=&O(\delta) D_{-\frac{1}{2}} \left(\Phi_x^2+\Psi_x^2+W_x^2+Y^2+{|\bar u_x|^2+|\bar u|^4+\theta_1^2+\theta_3^2}\right)\\
	&+{O(1)\left[\left(\Phi_x+W_x+Y+|\bar u|^2+\theta_1+\theta_3\right)\Psi_{xx}+\left(\Phi_{xx}+W_{xx}+Y_x\right)\Psi_x\right]}.\\%W_{xx}\Psi_x+Y_x\Psi_x+W_x\Psi_{xx}+Y\Psi_{xx}+\Phi_{xx}\Psi_{x}+\Psi_{xx}^2+|\tilde u|^4+\theta_1^2+\theta_3^2\right)}\\%+|\tilde{u}|^4+\theta_1^2+\theta_3^2\right)\\
	%&\blue{+O(1)\left(\Phi_x+\Psi_x+W_x+Y+|\tilde{u}|^2+|\theta_1|+|\theta_3|\right)\left(\Phi_{xx}+\Psi_{xx}+W_{xx}+Y_x+|\tilde u\tilde{u}_x|+|\theta_{1x}|+|\theta_{3x}|\right)}. %\left(\Phi_x\Phi_{xx}+\Psi_x\Psi_{xx}+W_xW_{xx}+YY_x\right).\\
	\end{aligned}
\end{align}

In the following subsection, we will consider the Cauchy problem for the reformulated system \eqref{sys-pert}. Because the local existence of \cref{sys-pert} is well known, we omit it for brevity. To prove the global existence, %\cref{mt}
 we present the following a priori assumptions,
\begin{align}\label{apa}
N(T)=\sup _{0 \leqslant t \leqslant T}\left\{\|(\Phi, \Psi, W)\|_{L^{\infty}}^2+(1+t)^{\frac{1}{2}}\|(\phi, \psi, \zeta)\|^2+(1+t)^{\frac{3}{2}}\|(\phi_x,\phi_{xx}, \psi_x, \zeta_x)\|^2\right\} \leqslant \chi^2,
\end{align}
where $\chi>0$ denotes a small constant depending on $\varepsilon$ and $\delta$. {By \eqref{Z}, it is obvious that %$\left|\bar{\theta}_1\right|+\left|\bar{\theta}_3\right| \leqslant C \varepsilon_0$
\begin{align}\label{smallconstants}
{\left|\bar{\theta}_1\right|+\left|\bar{\theta}_3\right| \leqslant C\varepsilon_0,}%\bar\theta_i\le\e_0.
\end{align}
 for some constant $C>0$.} For simplicity, we denote $\bar{\delta}:=\varepsilon+\chi+\delta.$
\subsection{Diagonalized system}
%\red{Note that if we use standard classical energy mothed, the term $Q_1\Psi$ contains $(1+t)^{-1}\Psi$ which cannot be controlled by the dissipation from the viscosity and heat conductivity. In order to overcome this difficulty, the diagonalized system is constructed in this subsection.}
Denoting $\mathcal{W}=(\Phi, \Psi, W)^t$, one has,
\begin{align}\label{equ-W}
\mathcal{W}_t+A_1 \mathcal{W}_x=A_2 \mathcal{W}_{x x}+A_3,
\end{align}
where
$$
\begin{gathered}
	A_1=\left(\begin{array}{ccc}
		0 & -1 & 0 \\
		-\frac{p_{+}}{\bar{v}} & 0 & \frac{R}{\bar{v}} \\
		0 & \frac{\gamma-1}{R} p_{+} & 0
	\end{array}\right), \quad A_2=\left(\begin{array}{ccc}
		0 & 0 & 0 \\
		0 & \frac{\mu}{\bar{v}} & 0 \\
		0 & 0 & \frac{\kappa(\gamma-1)}{R \bar{v}}
	\end{array}\right), \quad
	A_3=\left(\tilde{R}_1, Q_1,  Q_2\right)^t .
\end{gathered}
$$
The eigenvalues of the matrix $A_1$ are $\lambda_1, 0, \lambda_3$, where $\lambda_3=-\lambda_1=\sqrt{\frac{\gamma p_{+}}{\bar{v}}}$. And the corresponding left and right eigenvectors are
$$
\begin{aligned}
	& l_1=\sqrt{\frac{1}{2 \gamma}}\left(-1,-\frac{\gamma}{\lambda_3}, \frac{R}{p_{+}}\right), \quad l_2=\sqrt{\frac{\gamma-1}{\gamma}}\left(1,0, \frac{R}{(\gamma-1) p_{+}}\right), \\
	& l_3=\sqrt{\frac{1}{2 \gamma}}\left(-1, \frac{\gamma}{\lambda_3}, \frac{R}{p_{+}}\right), \quad r_1=\sqrt{\frac{1}{2 \gamma}}\left(-1,-\lambda_3, \frac{\gamma-1}{R} p_{+}\right)^t, \\
	& r_2=\sqrt{\frac{\gamma-1}{\gamma}}\left(1,0, \frac{p_{+}}{R}\right)^t, \quad r_3=\sqrt{\frac{1}{2 \gamma}}\left(-1, \lambda_3, \frac{\gamma-1}{R} p_{+}\right)^t.
\end{aligned}
$$
Direct calculations yield,
$$
l_i r_j=\delta_{i j}, i, j=1,2,3, \quad L A_1 R=\Lambda=\left(\begin{array}{ccc}
	\lambda_1 & 0 & 0 \\
	0 & 0 & 0 \\
	0 & 0 & \lambda_3
\end{array}\right) \text {, }
\quad 
\text{where}\quad
L=\left(l_1, l_2, l_3\right)^t, R=\left(r_1, r_2, r_3\right) .
$$
Let
$$
B=L \mathcal{W}=\left(b_1, b_2, b_3\right)^t,
$$
then multiplying \cref{equ-W} by  $L$, one has
\begin{align}\label{equ-B}
\begin{aligned}
	B_t+\Lambda B_x=L A_2 R B_{x x}+2 L A_2 R_x B_x+\big[\left(L_t+\Lambda L_x\right) R+L A_2 R_{x x}\big] B+L A_3 .
\end{aligned}
\end{align}
The viscosity matrix are
$$
L A_2 R=A_4=\left(\begin{array}{lll}
	b_{11} & b_{12} & b_{13} \\
	b_{12} & b_{22} & b_{12} \\
	b_{13} & b_{12} & b_{11}
\end{array}\right),
$$
with
$$
\begin{aligned}
	& \bar{v} b_{11}=\frac{\mu}{2}+\frac{(\gamma-1)^2 \kappa}{2 \gamma R}%=\bar vb_{33}
	, \quad \bar{v} b_{12}=\sqrt{\frac{\gamma-1}{2}} \frac{(\gamma-1) \kappa}{\gamma R}%=\bar vb_{21}=\bar vb_{23}=\bar vb_{32}
	, \\
	& \bar{v} b_{13}=-\frac{\mu}{2}+\frac{(\gamma-1)^2 \kappa}{2 \gamma R}%=\bar vb_{31}
	, \quad \bar{v} b_{22}=\frac{(\gamma-1) \kappa}{\gamma R} .
\end{aligned}
$$
Applying $\px^k$, $k=0,1,2$ to \cref{equ-B}, one has
\begin{align}\label{equ-Bk}
	\begin{aligned}
		\px^kB_t+\Lambda \px^kB_x=L A_2 R \px^kB_{x x}+\mathcal{M}_k,
	\end{aligned}
\end{align}
where
\begin{align}
	\begin{aligned}
	\mathcal{M}_k:=&\sum_{j=1}^{k}\bigg[-\px^j\Lambda\px^{k-j+1}B+\px^j(LA_2R)\px^{k-j+2}B\bigg]+2\sum_{i=0}^{k}\px^i\big(L A_2 R_x\big) \px^{k-i+1}B\\
	&+\sum_{i=0}^{k}\bigg\{\px^i\big[\left(L_t+\Lambda L_x\right) R+L A_2 R_{x x}\big] \px^{k-i}B+\px^{i}L \px^{k-i}A_3 \bigg\}\\
	\le&O(1)\sum_{j=1}^kD_{-\frac{j}{2}}%\red{\px^{k-j+1}B}
	{\big(\abs{\px^{k-j+1}b_1}+\abs{\px^{k-j+1}b_3}\big)}+\sum_{j=1%0
	}^{k+2}|D_{-\frac{j}{2}}\px^{k-j+2}B|+\sum_{i=0}^{k}|\px^{i}L \px^{k-i}A_3|\\
	:=&\sum_{i=1}^{3}|\mathcal{M}_{k}^{(i)}|.
\end{aligned}
\end{align}

\subsection{Energy estimates}\label{Ees1}
For the global existence, we refer to \cite{HMX,HXY}. In this paper, we focus on the decay rates.  %To prove the global existence of solution in Theorem \ref{mt}, it is only need to close the a priori estimates as follows. %The a priori estimate is
\begin{Prop} 
	Under the same assumptions as \cref{mt}, let $(\Phi,\Psi,W,\zeta)$ be the solutions of \cref{sys-pert} in $[0,t]$ satisfying \cref{apa}. Then it holds that
	\begin{align}
\|(\Phi, \Psi, W)\|_{L^{\infty}}^2+(1+t)^{\frac{1}{2}}\|(\phi, \psi, \zeta)\|_{L^2}^2+(1+t)^{\frac{3}{2}}\|(\phi_x,\phi_{xx}, \psi_x, \zeta_x)\|_{L^2}^2\leq O(\varepsilon_0^2+\delta_0^{\frac{1}{2}}).
\end{align}
\end{Prop}
\begin{proof}

 {\bf Step 1.} We assume that $\hat\Theta_x>0$. The case when $\hat\Theta_x<0$ can be discussed similarly. Let $v_1=\hat\Theta / \theta_{+}$, then $\left|v_1-1\right| \leqq C \delta$. Multiplying \cref{equ-Bk} by $\bar{B}^{(k)}=\left(v_1^n \px^kb_1, \px^kb_2, v_1^{-n} \px^kb_3\right)$ with a large positive integer $n:=[\delta^{-\frac{1}{2}}]+1%>\delta^{-\frac{1}{2}}
 $, we obtain
\begin{align}
\begin{aligned}
	&\int_\R\left(\frac{v_1^n}{2} |\px^kb_1|^2+\frac{1}{2} |\px^kb_2|^2+\frac{v_1^{-n}}{2} |\px^kb_3|^2\right)_t+\bar{B}^{(k)}_x A_4 \px^{k+1}Bdx
	+\int_{\R} a_1|\px^kb_1|^2+ a_3|\px^kb_3|^2dx\\
	=&\int_{\R}\red-\bar{B}^{(k)}A_{4x}\px^{k}B_x+\bigg[\left(\frac{v_1^n}{2}\right)_t \abs{\px^kb_1}^2+\left(\frac{v_1^{-n}}{2}\right)_t \abs{\px^kb_3}^2\bigg]+\bar{B}^{(k)}\mathcal{M}_kdx:=I_1+I_2+I_3,
\end{aligned}
\end{align}
where 
\begin{align}
\begin{aligned}
	&a_1:=-\frac{v_1^{n-1}}{2}\left(n \lambda_1 v_{1 x}+v_1 \lambda_{1 x}\right)=O(\delta^{-\frac{1}{2}})D_{-\frac{1}{2}},\\
	& a_3:=\frac{v_1^{-n-1}}{2}\left(n \lambda_3 v_{1 x}-v_1 \lambda_{3 x}\right)=O(\delta^{-\frac{1}{2}})D_{-\frac{1}{2}}.
	\end{aligned}
\end{align} 
We should first verify the dissipation term, direct calculations yield
\begin{align}\label{ds}
	\begin{aligned}
	\px^{k+1}B^tA_4\px^{k+1}B=&\frac{\kappa(\gamma-1)}{R\bar{v}}\left[\sqrt{\frac{\gamma-1}{2\gamma}}(\px^{k+1}b_{1}+\px^{k+1}b_{3})+\frac{1}{\sqrt{\gamma}}\px^{k+1}b_{2}\right]^2\\
	&\qquad+\frac{\mu}{\bar{v}}\left[\frac{1}{\sqrt{2}}(\px^{k+1}b_{3}-\px^{k+1}b_{1})\right]^2\\
	=&O(1)\big(\abs{\px^{k+1}\Psi}^2+\abs{\px^{k+1}W}^2\big).
\end{aligned}
\end{align}
For the sake of convenience, we denote
\begin{align}\label{EKG}
\begin{aligned}
	&\tilde{E}_k:=\int_{\R}\frac{v_1^n}{2} |\px^kb_1|^2+\frac{1}{2}| \px^kb_2|^2+\frac{v_1^{-n}}{2} |\px^kb_3|^2dx,\quad \tilde{K}_k:=\int_{\R}\px^{k+1}B^tA_4\px^{k+1}Bdx,\\
	%\quad
	& G_k:=\int_{\R} a_1|\px^kb_1|^2+ a_3|\px^kb_3|^2dx.
	\end{aligned}
\end{align}
Note that $\norm{\Phi_x}_{H^2}^2$ is not include in $\tilde{K}_i$, $i=0,1,2$, we further denote
\begin{align}\label{Ki}
	K_i:=\int_{\R}\abs{\px^{i+1}b_1}^2+\abs{\px^{i+1}b_2}^2+\abs{\px^{i+1}b_3}^2dx=O(1)\int_{\R}\abs{\px^{i+1}\Phi}^2+\abs{\px^{i+1}\Psi}^2+\abs{\px^{i+1}W}^2dx.
\end{align} 
Obviously
\begin{align}\label{I12}
	I_1+I_2\leq C\bar{\delta}^{\frac{1}{2}}(1+t)^{-1}\tilde{E}_k+C\bar{\delta}\big(\tilde{K}_k+\norm{\px^{k+1}\Phi}^2\big),
\end{align}
note that
$$
\begin{aligned}
	\left|\int(\bar{B}^{(k)}-\px^kB)_x A_4 \px^{k}B_x d x\right| & \leq  C\bar{\delta} \int\left|\px^{k+1}B\right|^2 d x+Cn%O(\bar{\delta}^{-\frac{1}{2}})
	 \int|\px^{k}B|^2%\left(
	\left|\hat\Theta_x\right|^2 %+\left|\hat\Theta_{xx}\right|^2\right)
	d x \\
	& \leq C\bar{\delta}^{\frac{1}{2}}(1+t)^{-1} \tilde{E}_k+C\bar{\delta}^{\frac{1}{2}}\tilde{K}_k+C\bar{\delta}^{\frac{1}{2}} \int\left|\px^{k+1}\Phi\right|^2 d x.
\end{aligned}
$$
For $I_3$, one has
\begin{align}	\label{I31}
	\begin{aligned}
	\int_{\R}\bar{B}^{(k)}\mathcal{M}_{k}^{(1)}dx\leq& C\bar\delta^{\frac{1}{2}}\sum_{i=0}^k(1+t)^{-i}\int_{\R}a_1\abs{\px^{k-i}b_1}^2+a_3\abs{\px^{k-i}b_3}^2dx,\\
		\int_{\R}\bar{B}^{(k)}\mathcal{M}_{k}^{(2)}dx\leq& C\bar{\delta}%^{\frac{1}{2}}
		\sum_{i=1%0
		}^{k+1}(1+t)^{-i}\norm{\px^{k-i+1}B}^2.
	\end{aligned}
\end{align}
We next estimate the rest terms in $I_3$ more carefully. For $k=0$,
\begin{align}
	\int_{\R}\bar{B}^{(0)}L A_3dx\leq O(1)\int_{\R}|\bar{B}^{(0)}|\left(D_{-1}+\abs{Q_1}+\abs{Q_2}\right)dx,
\end{align}
and one has
\begin{align}\label{B01}
	\begin{aligned}
	\int_{\R}|\bar{B}^{(0)}|D_{-1}dx\leq C\bar{\delta}(1+t)^{-\frac{1}{2}}+C\bar{\delta}(1+t)^{-1}\tilde{E}_0.
	\end{aligned}
\end{align}
On the other hand
\begin{align}
\int\left|Q_1\right||B^{(0)}| d x \leqslant \int\left|\left(\frac{\mu}{v}-\frac{\mu}{\tilde{v}}\right) u_x+J_1+\frac{R}{\tilde{v}} Y\right||B^{(0)}| d x+\int\left|\tilde{R}_2\right||B^{(0)}| d x=: I_3^1+I_3^2.
\end{align}
Since
\begin{align}\label{B00}
	\begin{aligned}
&\int\left|J_1 \right|\left| \bar{B}^{(0)}\right| d x \leqslant  C\bar\delta%O(\varepsilon)
\left(\left\|\Phi_x\right\|^2+\tilde{K}_0\right)+C\bar{\delta}(1+t)^{-1} \tilde{E}_0+C\bar{\delta}(1+t)^{{-\frac{1}{2}}}%\blue{-\frac{5}{2}}}
,\\
	& \int\left|\left(\frac{\mu}{v}-\frac{\mu}{\tilde{v}}\right) u_x\right||\bar{B}^{(0)}| +|Y||\bar{B}^{(0)}| d x  \leqslant C\bar\delta%\varepsilon
	 \tilde{K}_0+C\bar{\delta}\left\|\Phi_x\right\|^2+C\bar\delta%\varepsilon
	 \tilde{K}_1+C\bar{\delta}(1+t)^{-1} \tilde{E}_0,
\end{aligned}
\end{align}
we obtain
\begin{align}\label{B02}
I_3^1 \leqslant C\bar{\delta}\left(\left\|\Phi_x\right\|^2+\tilde{K}_0\right)+C\bar{\delta}(1+t)^{-1} \tilde{E}_0+C\bar\delta%\varepsilon
\tilde{K}_1+C\bar{\delta}(1+t)^{{-\frac{1}{2}}}.%\blue{-\frac{5}{2}}}.
\end{align}
By \eqref{eqRi}, $I_3^2$ is similar to \cref{B01},
\begin{align}
I_3^2 \leqslant C\bar{\delta}(1+t)^{-1} \tilde{E}_0+C\bar{\delta}(1+t)^{{-\frac{1}{2}}}.%\blue{-\frac{5}{2}}}.%O(\bar{\delta})\left(\left\|\Phi_x\right\|^2+\tilde{K}_0\right)+O(\bar{\delta})(1+t)^{-1} \tilde{E}_0+O(\bar\delta%\varepsilon
%)\tilde{K}_1+O(\bar{\delta})(1+t)^{\red{-\frac{1}{2}}\blue{-\frac{5}{2}}}.
\end{align}
 And $Q_2$ is similar to \cref{B00}. Then by \cref{B01,B02}, we arrive at
\begin{align}\label{I32}
	\int_{\R}\bar{B}^{(0)}LA_3dx\leqslant C\bar{\delta}\left(\left\|\Phi_x\right\|^2+\tilde{K}_0+\tilde{K}_1\right)+C\bar{\delta}(1+t)^{-1} \tilde{E}_0+C\bar{\delta}(1+t)^{-\frac{1}{2}}.
\end{align}
For $k=1$,
\begin{align}\label{B11}
	\int_{\R}\bar{B}^{(1)}(\px L A_3+L\px A_3)dx\leq O(1)\int_{\R}\left|\bar{B}^{(1)}\right|\left(D_{-\frac{3}{2}}+D_{-\frac{1}{2}}\abs{Q_1}%+\abs{Q_2}
	+Q_{1x}+Q_{2x}\right)dx.
\end{align}
%then
 We can estimate
\begin{align}\label{B12}
	\begin{aligned}
	&\int_{\R}\left|\bar{B}^{(1)}\right|\left(D_{-\frac{3}{2}}+D_{-\frac{1}{2}}\tilde R_2\right)dx=C\bar{\delta}(1+t)^{-\frac{3}{2}}+C\bar{\delta}(1+t)^{-1}\tilde{E}_1,\\
	&\int_{\R}D_{-\frac{1}{2}}\left|J_1 \| \bar{B}^{(1)}\right| d x %\blue{\leqslant O(\bar{\delta})\sum_{i=0,1}(1+t)^{-(2-i)}\tilde{E}_{i}+ O(\bar{\delta})(1+t)^{-\frac{7}{2}}+\norm{\px B}^4_{L^4}}\\
	{%\qquad\qquad\qquad\qquad
	\leq C\bar{\delta}\sum_{i=0,1}(1+t)^{-(2-i)}\tilde{E}_{i}+ C\bar{\delta}(1+t)^{-\frac{3}{2}}+C\bar\delta%\varepsilon_0
	{K}_1,}\\
 & 	\int_{\R}D_{-\frac{1}{2}}\bigg(\left|\left(\frac{\mu}{v}-\frac{\mu}{\tilde{v}}\right) u_x\right|+|Y|\bigg)|\bar{B}^{(1)}| d x  \leqslant %O\left(\bar{\delta}+\varepsilon_0\right)
 C\bar\delta{K}_1%\blue{+\varepsilon_0K_2}
 +C\bar{\delta}(1+t)^{-1}\left( \tilde{E}_1+(1+t)^{-1}\tilde{E}_0\right).%\\
	%&{\int_{\R}\left|\bar{B}^{(1)}\right|D_{-\frac{1}{2}}|Q_1|dx\le C\bar{\delta}(1+t)^{-1}( \tilde{E}_1+(1+t)^{-1}\tilde{E}_0)+C\bar{\delta}(1+t)^{-\frac{3}{2}}+C\bar{\delta}%(1+t)^{-\frac{3}{2}}
	%K_1},
	\end{aligned}
\end{align}
Then we have %where %we have used the G-N inequality, then 
\begin{align}
{\int_{\R}\left|\bar{B}^{(1)}\right|D_{-\frac{1}{2}}|Q_1|dx\le C\bar{\delta}(1+t)^{-1}( \tilde{E}_1+(1+t)^{-1}\tilde{E}_0)+C\bar{\delta}(1+t)^{-\frac{3}{2}}+C\bar{\delta}%(1+t)^{-\frac{3}{2}}
	K_1}.
\end{align}
 For the derivatives, one has
\begin{align}\label{B13}
	\begin{aligned}
	\int_{\R}\abs{\bar{B}^{(1)}}\abs{Q_{1x}}dx=&\int_{\R}\abs{\bar{B}^{(1)}}\abs{\px\left(\frac{\mu}{v}-\frac{\mu}{\tilde{v}}\right) u_x+\left(\frac{\mu}{v}-\frac{\mu}{\tilde{v}}\right) u_{xx}+J_{1x}+\px(\frac{R}{\tilde{v}} Y)+\px\tilde{R}_2}dx\\
	\leq&C\bar\delta%\varepsilon_0
	(K_1{+K_2})%+\red{C\bar\delta\|\Psi_{xxx}\|^2}
	+C\bar{\delta}(1+t)^{-1}\big(\tilde{E}_1+(1+t)^{-1}\tilde{E}_0\big)+C\bar{\delta}(1+t)^{-\frac{3}{2}}.
	\end{aligned}
\end{align}
Estimates on $Q_{2x}$ are similar.
\begin{comment} 
\begin{align}\label{BQ213}
	\begin{aligned}
	\int_{\R}\abs{\bar{B}^{(1)}}\abs{Q_{2x}}dx%=&\int_{\R}\abs{\bar{B}^{(1)}}\abs{\px\left(\frac{\mu}{v}-\frac{\mu}{\tilde{v}}\right) u_x+\left(\frac{\mu}{v}-\frac{\mu}{\tilde{v}}\right) u_{xx}+J_{1x}+\px(\frac{R}{\tilde{v}} Y)+\px\tilde{R}_2}dx\\
	\leq&C\bar\delta%\varepsilon_0
	(K_1+%\tilde
	 K_2)+C\bar{\delta}(1+t)^{-1}\big(\tilde{E}_1+(1+t)^{-1}\tilde{E}_0\big)+C\bar{\delta}(1+t)^{-\frac{3}{2}}.
	\end{aligned}
\end{align}
\end{comment}
Then by \cref{B11,B12,B13}, one has
\begin{align}\label{I33}
\begin{aligned}
	\int_{\R}&\bar{B}^{(1)}(\px L A_3+L\px A_3)dx\\
	&\leq C\bar\delta%\varepsilon_0+\delta
	(K_1+%\tilde
	 K_2)+C\bar{\delta}(1+t)^{-1}\big(\tilde{E}_1+(1+t)^{-1}\tilde{E}_0\big)+C\bar{\delta}(1+t)^{-\frac{3}{2}}.
	 \end{aligned}
\end{align}
For $k=2$, one has
\begin{align}
	\int_{\R}\px\bar{B}^{(2)}(\px L A_3+L\px A_3)dx\leq O(1)\int_{\R}\px\bar{B}^{(2)}\left(D_{-\frac{3}{2}}+D_{-\frac{1}{2}}%\big(
	\abs{Q_1}%\blue{+\abs{Q_2}}\big)
	+Q_{1x}+Q_{2x}\right)dx.
\end{align}
We only focus on the highest order of derivatives, which involves $Q_{1x},\ Q_{2x}$. It holds 
\begin{align}\label{B23}
	\begin{aligned}
		&\int_{\R}\abs{\px\bar{B}^{(2)}}\abs{Q_{1x}}dx\\
		=&\int_{\R}\abs{\px\bar{B}^{(2)}}\abs{\px\left(\frac{\mu}{v}-\frac{\mu}{\tilde{v}}\right) u_x+\left(\frac{\mu}{v}-\frac{\mu}{\tilde{v}}\right) u_{xx}+J_{1x}+\px(\frac{R}{\tilde{v}} Y)+\px\tilde{R}_2}dx\\
		\leq&\int_{\R}\abs{\px\bar{B}^{(2)}}\abs{\Phi_{xx}(\Psi_{xx}+D_{-1})+\Phi_{x}(\Psi_{xxx}+D_{-\frac{3}{2}}) +J_{1x}+D_{-\frac{1}{2}}Y+Y_x+\px\tilde{R}_2}dx\\
		\leq&{C\bar\delta K_2}+%\blue{O(\varepsilon_0)K_3+}%O(\delta+\varepsilon_0+\chi)
		C\bar\delta\sum_{i=0}^{2}(1+t)^{-(3-i)}\tilde{E}_{i}
		+C\bar{\delta}(1+t)^{-\frac{5}{2}},
	\end{aligned}
\end{align}
similarly,
\begin{align}\label{BQ223}
	\begin{aligned}
		\int_{\R}\abs{\px\bar{B}^{(2)}}\abs{Q_{2x}}dx%=&\int_{\R}\abs{\px\bar{B}^{(2)}}\abs{\px\left(\frac{\mu}{v}-\frac{\mu}{\tilde{v}}\right) u_x+\left(\frac{\mu}{v}-\frac{\mu}{\tilde{v}}\right) u_{xx}+J_{1x}+\px(\frac{R}{\tilde{v}} Y)+\px\tilde{R}_2}dx\\
		%\leq&\int_{\R}\abs{\px\bar{B}^{(2)}}\abs{\Phi_{xx}(\Psi_{xx}+D_{-1})+\Phi_{x}(\Psi_{xxx}+D_{-\frac{3}{2}}) +J_{1x}+D_{-\frac{1}{2}}Y+Y_x+\px\tilde{R}_2}dx\\
		\leq&{C\bar\delta(1+t)^{-3}\tilde E_0}+%O(\varepsilon_0)K_3+
		%O(\delta+\varepsilon_0+\chi)
		C\bar\delta\sum_{i=1}^{3}(1+t)^{-(i-1)}\tilde{E}_{4-i}%\sum_{i=0}^{2}(1+t)^{-(3-i)}\tilde{E}_{i}
		+O(\bar{\delta})(1+t)^{-\frac{5}{2}},
	\end{aligned}
\end{align}
where we have used the a priori assumption $\norm{\bar{B}^{(1)}}^2\leq \chi(1+t)^{-1}$% and the G-N inequality
. Then we arrive at
\begin{align}\label{I34}
		\int_{\R}\px\bar{B}^{(2)}(\px L A_3+L\px A_3)dx\leq %\blue{O(\varepsilon_0)K_3+}
		{C\bar\delta(1+t)^{-3}\tilde E_0}+%O(\delta+\varepsilon_0+\chi)
		C\bar{\delta}\sum_{i=1}^{3}(1+t)^{-(i-1)}\tilde{E}_{4-i}+C\bar{\delta}(1+t)^{-\frac{5}{2}}.
\end{align}
Collecting \cref{I12,I31,I32,I33,I34}, one has
\begin{align}
	&\frac{d}{dt}\big(\sum_{i=0}^2\tilde{E}_i\big)+\sum_{i=0}^2(\tilde{K}_i+G_{i})\leq C\bar{\delta}^{\frac{1}{2}}(1+t)^{-1}\big(\sum_{i=0}^2\tilde{E}_i\big)+C\bar{\delta}(1+t)^{-\frac{1}{2}}+C\bar{\delta}^{\frac{1}{2}}\norm{\Phi_x}^2\nonumber\\
	&\qquad\qquad\qquad\qquad\qquad\qquad\qquad+C\bar\delta^{\frac{1}{2}}\norm{\Phi_{xx}}^2+C\bar\delta^{\frac{1}{2}}K_2%\tilde E_3
	,\label{1}\\
	&\frac{d}{dt}\big(\sum_{i=1}^2\tilde{E}_i\big)+\sum_{i=1}^2(\tilde{K}_i+G_{i})\leq C\bar{\delta}^{\frac{1}{2}}(1+t)^{-1}\big(\sum_{i=1}^2\tilde{E}_i+G_0\big)+C\bar{\delta}(1+t)^{-\frac{3}{2}}\nonumber\\
	&\qquad\qquad\qquad\qquad\qquad\qquad\qquad+C\bar{\delta}^{\frac{1}{2}}(1+t)^{-2}\tilde{E}_0+C\bar{\delta}^{\frac{1}{2}}\norm{\Phi_{xx}}^2{+C\bar\delta^{\frac{1}{2}}K_2}%\tilde E_3}
	,\label{2}\\
	&\frac{d}{dt}\tilde{E}_2+\tilde{K}_2+G_{2}\leq C\bar{\delta}^{\frac{1}{2}}\sum_{i=0}^%2
	{2}(1+t)^{-(3-i)}\tilde{E}_{i}+C\bar{\delta}(1+t)^{-\frac{5}{2}}+C\bar\delta^{\frac{1}{2}}K_2%\blue{+C\bar{\delta}\norm{\Phi_{xxx}}^2}
	%\nonumber\\
	%&\qquad\qquad\qquad\qquad
	\nonumber\\
	&\qquad\qquad\qquad\qquad\qquad\qquad\qquad+C\bar{\delta}^{\frac{1}{2}}\big[(1+t)^{-2}G_0+(1+t)^{-1}G_1\big].%\red{+C\bar\delta^{\frac{1}{2}}\tilde E_3}.
	\label{3}
\end{align}

{\bf Step 2.} We need to estimate $\norm{\px^{k}\Phi}^2$, $k=1,2,3$. {Taking $\p_x\cref{sys-pert}_{1}\times\frac{\mu}{\tilde v}\Phi_x-\cref{sys-pert}_2\times\Phi_x$}, 
%\blue{Multiplying \cref{sys-pert}$_2$ by $\Phi_x$,} 
one has
$$
\left(\frac{\mu}{2 \tilde{v}} \Phi_x^2\right)_t-\left(\frac{\mu}{2 \tilde{v}}\right)_t \Phi_x^2-\Phi_x \Psi_t+\frac{p_{+}}{\tilde{v}} \Phi_x^2=\left(\frac{R}{\tilde{v}} W_x-Q_1-\frac{\mu}{\tilde{v}} \tilde{R}_{1 x}\right) \Phi_x .
$$
Since
$$
\Phi_x \Psi_t=\left(\Phi_x \Psi\right)_t-\left(\Phi_t \Psi\right)_x+\Psi_x^2-\tilde{R}_1 \Psi_x,
$$
by \eqref{smallconstants} and \eqref{ansatz}, we obtain
\begin{align}\label{phi1}
\begin{aligned}
	& \left(\int \frac{\mu}{2 \tilde{v}} \Phi_x^2-\Phi_x \Psi d x\right)_t+\int \frac{p_{+}}{2 \tilde{v}} \Phi_x^2 d x  \leqslant C \int\left(\Psi_x^2+W_x^2\right) d x+C \int Q_1^2 d x+C\bar{\delta}(1+t)^{-\frac{3}{2}},
\end{aligned}
\end{align}
where by \cref{Q}, we have
\begin{align}\label{Q12}
\int Q_1^2 d x %&\blue{\leqslant O(\varepsilon)\left(K_1+\left\|\Phi_x\right\|^2\right)+O(\bar{\delta})(1+t)^{-\frac{3}{2}}+O(\varepsilon)\left\|\psi_x\right\|^2}\\
&{\le C\bar\delta K_1+C\bar\delta\tilde K_0+C\bar\delta(1+t)^{-\frac{3}{2}}}.
\end{align}
Combining \cref{phi1,Q12}, one has
\begin{align}\label{phix}
\left(\int \frac{\mu}{2 \tilde{v}} \Phi_x^2-\Phi_x \Psi d x\right)_t+\int \frac{p_{+}}{4 \tilde{v}} \Phi_x^2 d x %&\blue{\leqslant C_1 K_1+C_1 \bar{\delta}(1+t)^{-3 / 2}+C_1 \varepsilon_0\left\|\psi_x\right\|^2}\\
&{\le C_1\bar\delta K_1+C_1 K_0+C_1\bar\delta(1+t)^{-\frac{3}{2}}},
\end{align}
where $C_1$ is a positive constant.
Next, {taking $\p_{xx}\cref{sys-Ori}_1\times\frac{\mu}{\tilde v}-\p_x\cref{sys-Ori}_2$, }%\blue{by \cref{sys-Ori}$_2$,}
 one has
\begin{align}\label{phi2}
\frac{\mu}{\tilde{v}} \phi_{x t}-\psi_t-(p-\tilde{p})_x=-\left(\frac{\mu}{\tilde{v}}\right)_x \psi_x-\left[\left(\frac{\mu}{v}-\frac{\mu}{\tilde{v}}\right) u_x\right]_x+\tilde{R}_{2 x}-\frac{\mu}{\tilde{v}} \tilde{R}_{1 x x}.
\end{align}
Multiplying \cref{phi2} by $\phi_x$, we get
$$
\begin{aligned}
	& \left(\frac{\mu}{2 \tilde{v}} \phi_x^2\right)_t-\left(\frac{\mu}{2 \tilde{v}}\right)_t \phi_x^2-\psi_t \phi_x-(p-\tilde{p})_x \phi_x =\left\{-\left(\frac{\mu}{\tilde{v}}\right)_x \psi_x+\left(\frac{\mu \Phi_x}{v \tilde{v}} u_x\right)_x+\tilde{R}_{2 x}-\frac{\mu}{\tilde{v}} \tilde{R}_{1 x x}\right\} \phi_x .
\end{aligned}
$$
Note that
$$
-(p-\tilde{p})_x=\frac{\tilde{p}}{\tilde{v}} \phi_x-\frac{R}{\tilde{v}} \zeta_x+\left(\frac{p}{v}-\frac{\tilde{p}}{\tilde{v}}\right) v_x-\left(\frac{R}{v}-\frac{R}{\tilde{v}}\right) \theta_x,
$$
then similar to \cref{phix}, we have
\begin{align}\label{phixx}
\begin{aligned}
	& \left(\int \frac{\mu}{2 \tilde{v}} \phi_x^2-\phi_x \psi d x\right)_t+\int \frac{\tilde{p}}{2 \tilde{v}} \phi_x^2 d x\\
	&  \leqslant C_2 K_1+C_2 \bar{\delta}(1+t)^{-1} \tilde{E}_1+C_2 \bar{\delta}(1+t)^{-\frac{5}{2}}+C_2 {\bar\delta}%\blue{\varepsilon_0}
	 \int \psi_{x x}^2 d x,
\end{aligned}
\end{align}
where constant $C_2>0$.

Finally, we need to estimate $\norm{\phi_{xx}}^2$ and handle the highest order of derivatives with care. {Taking $\p_{xxx}\cref{sys-Ori}_1\times\frac{\mu}{\tilde v}-\p_{xx}\cref{sys-Ori}_2$, } one has %\blue{
\begin{align}\label{phi3}
\begin{aligned}
\frac{\mu}{\tilde{v}} \phi_{xx t}-\psi_{xt}-(p-\tilde{p})_{xx}=&-\left(\frac{\mu}{\tilde{v}}\right)_{xx} \psi_x-{2}\left(\frac{\mu}{\tilde{v}}\right)_{x} \psi_{xx}-\left[\left(\frac{\mu}{v}-\frac{\mu}{\tilde{v}}\right) u_x\right]_{xx}\\
&+\left(\frac{\mu}{\tilde v}\right)_x\tilde R_{1xx}+\tilde{R}_{2 xx}-\big(\frac{\mu}{\tilde{v}} \tilde{R}_{1 x x}\big)_x.%\blue{-\left(\frac{\mu}{\tilde v}\right)_x\phi_{xt}}.
\end{aligned}
\end{align}
Multiplying \cref{phi3} by $\phi_{xx}$, we get
\begin{align}
\begin{aligned}
	& \left(\frac{\mu}{2 \tilde{v}} \phi_{xx}^2\right)_t-\left(\frac{\mu}{2 \tilde{v}}\right)_t \phi_{xx}^2-\psi_{xt} \phi_{xx}-(p-\tilde{p})_{xx} \phi_{xx} \\
	& \quad=\left\{-\left(\frac{\mu}{\tilde{v}}\right)_{xx} \psi_x-2\left(\frac{\mu}{\tilde{v}}\right)_{x} \psi_{xx}-\left[\left(\frac{\mu}{v}-\frac{\mu}{\tilde{v}}\right) u_x\right]_{xx}{+\left(\frac{\mu}{\tilde v}\right)_x\tilde R_{1xx}}+\tilde{R}_{2 xx}-\big(\frac{\mu}{\tilde{v}} \tilde{R}_{1 x x}\big)_x%\blue{-\left(\frac{\mu}{\tilde v}\right)_x\phi_{xt}}
	\right\} \phi_{x x} .
\end{aligned}
\end{align}
Note that
\begin{align}
	\begin{aligned}
-(p-\tilde{p})_{xx}=&\frac{\tilde{p}}{\tilde{v}} \phi_{xx}-\frac{R}{\tilde{v}} \zeta_{xx}+\big(\frac{\tilde{p}}{\tilde{v}}\big)_x \phi_{x}-\big(\frac{R}{\tilde{v}}\big)_x \zeta_{x}+\left(\frac{p}{v}-\frac{\tilde{p}}{\tilde{v}}\right)_x v_x\\
&-\left(\frac{R}{v}-\frac{R}{\tilde{v}}\right)_x \theta_x+\left(\frac{p}{v}-\frac{\tilde{p}}{\tilde{v}}\right) v_{xx}-\left(\frac{R}{v}-\frac{R}{\tilde{v}}\right)\theta_{xx},
	\end{aligned}
\end{align}
one has
\begin{align}\label{phixxx}
\begin{aligned}
	& \left(\int \frac{\mu}{2 \tilde{v}} \phi_{xx}^2-\phi_{xx} \psi_x d x\right)_t+\int \frac{\tilde{p}}{2 \tilde{v}} \phi_{xx}^2 d x  \\
	&\leqslant C_3 K_2+C_3 \bar{\delta}\bigg[(1+t)^{-1} \tilde{E}_2+(1+t)^{-2}\tilde{E}_{1}\bigg]+C_3 \bar{\delta}(1+t)^{-\frac{7}{2}},
\end{aligned}
\end{align}
where constant $C_3>0$.
For the highest order of derivatives, we have
\begin{align}
\begin{aligned}
	\int\left(\frac{\mu \phi}{v \tilde{v}} u_{x}\right)_{xx} \phi_{xx}d x &=	\int\left(\frac{\mu \phi}{v \tilde{v}}\right)_{xx} u_{x} \phi_{xx}+2\left(\frac{\mu \phi}{v \tilde{v}} \right)_{x}u_{xx} \phi_{xx}+\left(\frac{\mu \phi}{v \tilde{v}} \right)u_{xxx} \phi_{xx}d x ,\\
		\int\left(\frac{\mu \phi}{v \tilde{v}} \right)\psi_{xxx} \phi_{xx}d x &=\int\left(\frac{\mu \phi}{v \tilde{v}} \right)(\phi_{xxt}+\tilde{R}_{1xxx}) \phi_{xx}d x\\
		&=\int{\left(\frac{\mu\phi}{2v\tilde v}\phi_{xx}^2\right)_t-}\left(\frac{\mu \phi}{2v \tilde{v}} \right)_t{\phi^2_{xx}}+\left(\frac{\mu \phi}{v \tilde{v}} \right)\tilde{R}_{1xxx} \phi_{xx}d x,
\end{aligned}
\end{align}
where we have used %G-N inequality and the
 a $priori$ assumption ${\norm{\phi_x}_{L^\infty}\leq\sqrt 2\norm{\phi_x}_{L^2}^{\frac{1}{2}}\norm{\phi_{xx}}^{\frac{1}{2}}_{L^2}\leq \chi%^{\frac{1}{2}}
(1+t)^{-\frac{1}{2}}}.$ %$\red{\norm{\phi_{xx}}^{\frac{1}{2}}_{L^2}\leq \chi(1+t)^{-\frac{3}{4}}} $ $\blue{\norm{\phi_x}_{L^\infty}\leq \chi(1+t)^{-\frac{3}{4}}}$
	Obviously, it holds that %there exists %a positive constant $c$ such that
\begin{align}
	E_i:=\bar C_{i+1}\tilde{E}_{i}+%c
	\hat C_{i+1}\int_{\R}\frac{\mu}{2\tilde{v}}\abs{\px^{i}\phi}^2-%\blue{\px^{i}\Phi\Psi }
	{\p_x^i\phi\p_x^i\Psi}%\red{\p_x^i\left(\Phi_x\Psi\right)}
	dx>\int_{\R} \abs{\px^{i}b_1}^2+\abs{\px^{i}b_2}^2+\abs{\px^{i}b_3}^2+\abs{\px^{i}\phi}^2dx,%\quad {i=0,1,2.}
\end{align} 
where $i=0,1,2,$ and we choose suitably constants $\bar C_{i+1},\hat C_{i+1}$ satisfying
 $$\bar C_{i+1}>1,\ \ \hat C_{i+1}>1,\ %\ \ \frac{\bar C_1}{8}-C_1>0, \ % \bar C_3>1,
 \  \frac{1}{4}\bar C_{i+1}-\hat C_{i+1}C_{i+1}-\hat C_{i+1}>0%\frac{1}{4}\bar C_j
% ,\ \ j=2,3
 .$$
Then by \cref{phix,phixx,phixxx}, \cref{1,2,3} become
\begin{align}
	&\frac{d}{dt}\big(\sum_{i=0}^2E_i\big)+\sum_{i=0}^2(K_i+G_{i})\leq C\bar{\delta}^{\frac{1}{2}}(1+t)^{-1}\big(\sum_{i=0}^2{E}_i\big)+C\bar{\delta}(1+t)^{-\frac{1}{2}},\label{11}\\
	&\frac{d}{dt}\big(\sum_{i=1}^2{E}_i\big)+\sum_{i=1}^2(K_i+G_{i})\leq C\bar{\delta}^{\frac{1}{2}}(1+t)^{-1}\big(\sum_{i=1}^2{E}_i+G_0\big)+C\bar{\delta}(1+t)^{-\frac{3}{2}}+C\bar{\delta}^{\frac{1}{2}}(1+t)^{-2}{E}_0,\label{22}\\
	&\frac{d}{dt}{E}_2+{K}_2+G_{2}\leq C\bar{\delta}^{\frac{1}{2}}\sum_{i=0}^2(1+t)^{-(3-i)}{E}_{i}+C\bar{\delta}(1+t)^{-\frac{5}{2}}\nonumber\\
	&\qquad\qquad\qquad\qquad\qquad+C\bar{\delta}^{\frac{1}{2}}\big[(1+t)^{-2}G_0+(1+t)^{-1}G_1\big].\label{33}
\end{align}

%{\bf Step 3.}
 {\bf The proof of Theorem \ref{mt}.} 
In the following estimate, we use the fact that
\begin{align}
	\tilde{E}_1=O(1)K_0,\quad \tilde{E}_2=O(1)K_1.
\end{align}
By Gronwall's inequality and $\cref{11}\times(1+t)^{-C\bar\delta^{\frac{1}{2}}}$, one has
\begin{align}
\sum_{i=0}^2E_i\leq O(\varepsilon^2+\bar{\delta}^{\frac{1}{2}})(1+t)^{\frac{1}{2}},\qquad\int_{%{\blue\R}
0}^{ t}\sum_{i=0}^2(K_i+G_{i})d%\blue{x}
{\tau}\leq O(\varepsilon^{2}+\bar{\delta}^{\frac{1}{2}})(1+t)^{\frac{1}{2}}.
\end{align}
Multiplying \cref{22} by $(1+t)$ and then integrating on $[0,t]$, one has
\begin{align}
	(1+t)\sum_{i=1}^2E_i+\int_{0}^t(1+\tau)\sum_{i=0}^2(K_i+G_{i})d\tau\leq O(\varepsilon^2+\bar{\delta}^{\frac{1}{2}})(1+t)^{\frac{1}{2}}.
\end{align}
Then it follows that
\begin{align}\label{E1}
	\sum_{i=1}^2E_i\leq O(\varepsilon^{2}+\bar{\delta}^{\frac{1}{2}})(1+t)^{-\frac{1}{2}},\qquad\int_{0}^t(1+\tau)\sum_{i=1}^2(K_i+G_{i})d\tau\leq O(\varepsilon^2+\bar{\delta}^{\frac{1}{2}})(1+t)^{\frac{1}{2}}.
\end{align}
Finally, multiplying \cref{33} by $(1+t)^2$, one has
\begin{align}
	(1+t)^2E_2+\int_{0}^t(1+\tau)^2(K_2+G_{2})d\tau\leq O(\varepsilon^2+\bar{\delta}^{\frac{1}{2}})(1+t)^{\frac{1}{2}},
\end{align}
ans it holds that
\begin{align}\label{E2}
	E_2\leq O(\varepsilon^2+\bar{\delta}^{\frac{1}{2}})(1+t)^{-\frac{3}{2}}.
\end{align}
Then the desired decay rate can be obtained  by \cref{WZEta,E1,E2}
\begin{align}
	\norm{\phi,\psi,\zeta}_{L^\infty}\le\norm{\phi,\psi,\zeta}^{\frac{1}{2}}\norm{\phi_x,\psi_x,\zeta_x}^{\frac{1}{2}}\leq O(\varepsilon+\bar{\delta}^{\frac{1}{4}})(1+t)^{-\frac{1}{2}}.
\end{align}
Hence, we get the convergence rate of \eqref{L2decay}-\eqref{decay} and close the a priori estimates \eqref{apa}. The proof of Theorem \ref{mt} is finished.
\end{proof}

\section{%\blue{Energy estimates for the case of} 
Zero initial mass}
In this section, our purpose is to prove the Theorem \ref{mt0}. We first list the integrated system, and then show an important Poincar\'e type inequality about this perturbation system. %obtain the results of Theorem 2, 
By diagonalized system, we make the terms with poor decay rates disappear in the second equation, so that the bad terms are controlled by %similar %inequality as 
%\eqref{1.5} the 
%estimates of 
\begin{align}\label{1.5}
\int_{\mathbb R}\left|\frac{v_1^{-n}}{\lambda_3}\mathring b_3+
\frac{v_1^{n}}{\lambda_3}\mathring b_1\right|\left|{R}_1\right|dx\le C\bar\delta \mathring G_0+C\bar\delta(1+t)^{-1}.
\end{align}
 In contrast to the previous inequality \eqref{I32}, the decay rate here becomes $(1+t)^{-1}$, so we can get a better estimate. %Compare %Thus the energy estimate is closed.}
\subsection{Preliminaries}
By the introduction, we know that the quantities $(\mathring\Phi,\mathring\Psi,\mathring{\bar W})$ can be well defined in some Sobolev space through imposing $\mathring\Phi(\infty,0)=\mathring\Psi(\infty,0)=\mathring{\bar W}(\infty,0)=0.$ 
%Firstly,
 Direct calculations yield the following relationship,
$$
(\phib, \psib)=(\mathring\Phi, \mathring\Psi)_x\quad \text{and}\quad \frac{R}{\gamma-1} \mathring\zeta+\frac{1}{2}\left|\mathring\Psi_x\right|^2+\bar{u} \mathring\Psi_x=\mathring{\bar{W}}_x.
$$
%Then we study
Similar to the case for non-zero initial excess mass, subtracting \eqref{cd} from equation \eqref{1.1} and integrating the resulting system yield % we give 
the following integrated system,
\begin{align}\label{sys-Ori0}
\left\{\begin{array}{l}
	\mathring\Phi_t-\mathring\Psi_x=0%-\tilde{R}_1
	, \\
	\mathring\Psi_t+p-\bar{p}=\frac{\mu}{v} u_x-\frac{\mu}{\bar{v}} \bar{u}_x-{R}_1, \\
	\mathring{\bar{W}}_t+p u-\bar{p} \bar{u}=\frac{\kappa}{v} \theta_x-\frac{\kappa}{\bar{v}} \bar{\theta}_x+\frac{\mu}{v} u u_x-\frac{\mu}{\bar{v}} \bar{u} \bar{u}_x-{R}_2.
\end{array}\right.
\end{align}
%which is given by \cref{1.1,sys-ansatz}. 
To capture the diffusive effect, we introduce
\begin{align}\label{WZEta0}
\mathring W=\frac{\gamma-1}{R}(\mathring{\bar{W}}-\bar{u} \mathring\Psi),\qquad
\Rightarrow\qquad 
\mathring\zeta=\mathring W_x-\mathring Y, \quad \text { with\quad } \mathring Y=\frac{\gamma-1}{R}\left(\frac{1}{2} \mathring\Psi_x^2-\bar{u}_x \mathring\Psi\right).
\end{align}
Using the new variable $\mathring W$ and linearizing the left-hand side of the system \eqref{sys-Ori0}%(3.1)
, we have
\begin{align}\label{sys-pert00}
\left\{\begin{array}{l}
	\mathring\Phi_t-\mathring\Psi_x=0%-\tilde{R}_1
	, \\
	\mathring\Psi_t-\frac{p_{+}}{\bar{v}} \mathring\Phi_x+\frac{R}{\bar{v}} \mathring W_x=\frac{\mu}{\bar{v}} \mathring\Psi_{x x}+\mathring Q_1, \\
	\frac{R}{\gamma-1} \mathring W_t+p_{+} \mathring\Psi_x=\frac{\kappa}{\bar{v}} \mathring W_{x x}+\mathring Q_2,
\end{array}\right.
\end{align}
where
\begin{align}
&\begin{aligned}\label{J0}
	\mathring J_1=&\frac{\bar{p}-p_{+}}{\bar{v}} \mathring\Phi_x-\left[p-\bar{p}+\frac{\bar{p}}{\bar{v}} \mathring\Phi_x-\frac{R}{\bar{v}}(\theta-\bar{\theta})\right]=O(1)\left(\mathring\Phi_x^2+\mathring W_x^2+\mathring Y^2+|\bar{u}|^4%{+\theta_1^2+\theta_3^2}
	\right), \\
	\mathring J_2=&\left(p_{+}-p\right) \mathring\Psi_x=O(1)\left(\mathring\Phi_x^2+\mathring\Psi_x^2+\mathring W_x^2+\mathring Y^2+|\bar{u}|^4%{+\theta_1^2+\theta_3^2}
	\right), 
\end{aligned}\\
&\begin{aligned}\label{Q0}
	\mathring Q_1=&\left(\frac{\mu}{v}-\frac{\mu}{\bar{v}}\right) u_x+\mathring J_1+\frac{R}{\bar{v}} \mathring Y-{R}_1, \\
	\mathring Q_2=&\left(\frac{\kappa}{v}-\frac{\kappa}{\bar{v}}\right) \theta_x+\frac{\mu u_x}{v} \mathring\Psi_x-{R}_2-\bar{u}_t \mathring\Psi+\bar{u} {R}_1+\mathring J_2-\frac{\kappa}{\bar{v}} \mathring Y_x,
\end{aligned}
\end{align}
and %\blue{Also, one can obtain the estimates for derivatives}
\begin{align}\label{Jx0}
	\begin{aligned}
	{\mathring J_{1x}}=&O(\delta) D_{-\frac{1}{2}} \left(\mathring\Phi_x^2+\mathring W_x^2+\mathring Y^2+|\bar u_x|^2+|\bar u|^4%+\theta_1^2+\theta_3^2}%+|\tilde{u}|^4
	\right)+{O(1)\left[%\left(|\bar u|^2+\theta_1+\theta_3\right)
	|\bar u|^2\mathring\Phi_{xx}+\mathring\Phi_{xx}^2+\mathring W_{xx}^2+\mathring Y_x^2\right]},\\
	%{\mathring J_{1xx}}=&O(\delta)D_{-1}\left(\mathring\Phi_x^2+\mathring W_x^2+\mathring Y^2+|\bar u_x|^2%+|\bar u\bar v_x|^2+|\bar v_x|^4+|\bar v_{xx}|^2
	%+|\bar u|^4\right)+O(\delta)D_{-\frac{1}{2}} \left(\mathring\Phi_{xx}^2+\mathring W_{xx}^2+\mathring Y_x^2+|\bar u_x|\mathring\Phi_{xx}+|\bar u|^2\mathring\Phi_{xx}%+|\bar u_x|^2+|\bar u|^4%|\bar u\bar v_x|^2%+\theta_1^2+\theta_3^2}%+|\tilde{u}|^4
	%\right)\\
	%&+{O(1)\left[%\left(|\bar u|^2+\theta_1+\theta_3\right)
	%|\bar u|^2\mathring\Phi_{xxx}+\mathring\Phi_{xxx}^2+\mathring W_{xxx}^2+\mathring Y_{xx}^2+|\bar u_{xx}|^2\right]}\\
	%\Phi_{xx}+W_{xx}+Y_x%+|\tilde u\tilde u_x|
	%+|\tilde u|^4+\theta_1^2+\theta_3^2\right)},
	%&\blue{+O(1)\left(\Phi_x\Phi_{xx}+W_xW_{xx}+YY_x\right)},\\
	{\mathring J_{2x}}=&O(\delta) D_{-\frac{1}{2}} \left(\mathring\Phi_x^2+\mathring\Psi_x^2+\mathring W_x^2+\mathring Y^2+|\bar u_x|^2+|\bar u|^4%+\theta_1^2+\theta_3^2}
	\right)\\
	&+{O(1)\left[\left(\mathring\Phi_x+\mathring W_x+\mathring Y+|\bar u|^2%+\theta_1+\theta_3
	\right)\mathring\Psi_{xx}+\left(\mathring\Phi_{xx}+\mathring W_{xx}+\mathring Y_x\right)\mathring\Psi_x\right]}.%W_{xx}\Psi_x+Y_x\Psi_x+W_x\Psi_{xx}+Y\Psi_{xx}+\Phi_{xx}\Psi_{x}+\Psi_{xx}^2+|\tilde u|^4+\theta_1^2+\theta_3^2\right)}\\%+|\tilde{u}|^4+\theta_1^2+\theta_3^2\right)\\
	%&\blue{+O(1)\left(\Phi_x+\Psi_x+W_x+Y+|\tilde{u}|^2+|\theta_1|+|\theta_3|\right)\left(\Phi_{xx}+\Psi_{xx}+W_{xx}+Y_x+|\tilde u\tilde{u}_x|+|\theta_{1x}|+|\theta_{3x}|\right)}. %\left(\Phi_x\Phi_{xx}+\Psi_x\Psi_{xx}+W_xW_{xx}+YY_x\right).\\
	%{\mathring J_{2xx}}=&O(\delta)D_{-1}\left(\mathring\Phi_x\mathring\Psi_x+(\mathring W_x+\mathring Y)\mathring\Psi_x+\mathring\Psi_x^2+|\bar u_x|^2+|\bar u|^4\right)\\
	%&+O(\delta)D_{-\frac{1}{2}}\left[\left(\mathring\Phi_x+\mathring W_x+\mathring Y+|\bar u|^2+|\bar u_x|%+\theta_1+\theta_3
	%\right)\mathring\Psi_{xx}+\left(\mathring\Phi_{xx}+\mathring W_{xx}+\mathring Y_x\right)\mathring\Psi_x%+\mathring\Psi_{xx}^2+|\bar u_x|^2+|\bar u|^4
	%\right]\\ %\left(\mathring\Phi_x^2+\mathring\Psi_x^2+\mathring W_x^2+\mathring Y^2+|\bar u_x|^2+|\bar u|^4%+\theta_1^2+\theta_3^2}
	%%\right)\\
	%&+{O(1)\left[\left(\mathring\Phi_x+\mathring W_x+\mathring Y+|\bar u|^2%+\theta_1+\theta_3
	%\right)\mathring\Psi_{xxx}+\left(\mathring\Phi_{xx}+\mathring W_{xx}+\mathring Y_x\right)\mathring\Psi_{xx}+\left(\mathring\Phi_{xxx}+\mathring W_{xxx}+\mathring Y_{xx}\right)\mathring\Psi_x+|\bar u_{xx}|^2\right]}.
	\end{aligned}
\end{align}
\begin{comment}
\begin{align}
\begin{aligned}
{\mathring J_{1xx}}=&O(\delta)D_{-1}\left(\mathring\Phi_x^2+\mathring W_x^2+\mathring Y^2+|\bar u_x|^2%+|\bar u\bar v_x|^2+|\bar v_x|^4+|\bar v_{xx}|^2
	+|\bar u|^4\right)+O(\delta)D_{-\frac{1}{2}} \left(\mathring\Phi_{xx}^2+\mathring W_{xx}^2+\mathring Y_x^2+|\bar u_x|\mathring\Phi_{xx}+|\bar u|^2\mathring\Phi_{xx}%+|\bar u_x|^2+|\bar u|^4%|\bar u\bar v_x|^2%+\theta_1^2+\theta_3^2}%+|\tilde{u}|^4
	\right)\\
	&+{O(1)\left[%\left(|\bar u|^2+\theta_1+\theta_3\right)
	|\bar u|^2\mathring\Phi_{xxx}+\mathring\Phi_{xxx}^2+\mathring W_{xxx}^2+\mathring Y_{xx}^2+|\bar u_{xx}|^2\right]}\\
	{\mathring J_{2xx}}=&O(\delta)D_{-1}\left(\mathring\Phi_x\mathring\Psi_x+(\mathring W_x+\mathring Y)\mathring\Psi_x+\mathring\Psi_x^2+|\bar u_x|^2+|\bar u|^4\right)\\
	&+O(\delta)D_{-\frac{1}{2}}\left[\left(\mathring\Phi_x+\mathring W_x+\mathring Y+|\bar u|^2+|\bar u_x|%+\theta_1+\theta_3
	\right)\mathring\Psi_{xx}+\left(\mathring\Phi_{xx}+\mathring W_{xx}+\mathring Y_x\right)\mathring\Psi_x%+\mathring\Psi_{xx}^2+|\bar u_x|^2+|\bar u|^4
	\right]\\ %\left(\mathring\Phi_x^2+\mathring\Psi_x^2+\mathring W_x^2+\mathring Y^2+|\bar u_x|^2+|\bar u|^4%+\theta_1^2+\theta_3^2}
	%\right)\\
	&+{O(1)\left[\left(\mathring\Phi_x+\mathring W_x+\mathring Y+|\bar u|^2%+\theta_1+\theta_3
	\right)\mathring\Psi_{xxx}+\left(\mathring\Phi_{xx}+\mathring W_{xx}+\mathring Y_x\right)\mathring\Psi_{xx}+\left(\mathring\Phi_{xxx}+\mathring W_{xxx}+\mathring Y_{xx}\right)\mathring\Psi_x+|\bar u_{xx}|^2\right]}.
	\end{aligned}
	\end{align}
	\end{comment}

Similarly, we omit the local existence of the system \cref{sys-pert00} %Since the local existence of \cref{sys-pert} is well known, we omit it here
 for brevity. To prove \cref{mt0}, we present the following a priori assumptions,
\begin{align}\label{apa0}
\begin{aligned}
\mathring N(T)=\sup _{0 \leqslant t \leqslant T}\Big\{&\|(\mathring\Phi, \mathring\Psi, \mathring W)\|_{L^{\infty}}^2+(1+t)\ln^{-1}(t+2)%^{\frac{1}{2}}
\|(\phib, \psib, \mathring\zeta)\|^2\\
&+(1+t)\ln^{-1}(t+2)%^{\frac{3}{2}}
\|(\phib_x,\phib_{xx}, \psib_x, \mathring\zeta_x)\|^2\Big\} \leqslant {\chi^2},
\end{aligned}
\end{align}
where $\chi>0$ denotes a small constant depending on $\varepsilon$ and $\delta$. Motivated by \cite{HLM}, we define
\begin{align}\label{omega0}
\omega(x,t)=(1+t)^{-\frac{1}{2}}exp\left\{-\frac{\alpha x^2}{1+t}\right\}, \ \ g(x,t)=\int_{-\infty}^x\omega(y,t)dy,
\end{align}
where $\alpha>0$ will be determined later. By direct computation, it holds that
\begin{align}\label{heat1}
4\alpha g_t=\omega_x,\ \ \ \|g(\cdot,t)\|_{L^\infty}=\sqrt\pi\alpha^{-\frac{1}{2}}.
\end{align}
There is an elementary inequality concerning the heat kernel as follows.
\begin{Lem}\cite{HLM}\label{heat0}
For $0<T\le+\infty,$ assume that $h(x,t)$ satisfies 
\begin{align}
h_x\in L^2(0,T;L^2(\mathbb R)), \ \ h_t\in L^2(0,T;H^{-1}(\mathbb R)).
\end{align}
Then the following estimate holds:
\begin{align}
\begin{aligned}
\int_0^T\int_{\mathbb R}h^2\omega^2dxdt\le4\pi\|h(0)\|^2+4\pi\alpha^{-1}\int_0^T\|h_x(t)\|^2dt+8\alpha\int_0^T\langle{h_t,hg^2}\rangle_{H^{-1}\times H^1}dt.
\end{aligned}
\end{align}
\end{Lem}

Then we give the following Poincar\'e type inequality, which plays an important role in our energy estimates. %computations.
\begin{Lem}\cite{HLM,Yang1}\label{lem5}
For $\alpha\in(0,\frac{c_2}{4}]$ with $c_2$ as in \eqref{cde}, $\omega$ defined in \eqref{omega0}, and assume \eqref{apa0} holds. Then there exists some positive constant $C$ depending on $\alpha$ such that 
\begin{align}\label{poin8}
\int_0^t\int_{\mathbb R}\left(\mathring\Phi^2+\mathring\Psi^2+\mathring W^2\right)\omega^2dxd\tau\le C+C\int_0^t\|(\mathring\Phi,\mathring\Psi,\mathring W)_x\|^2d\tau{+C\int_0^t%\|\psib_x\|^2+\|\mathring\zeta_x\|
\|(\mathring\Psi,\mathring W)_{xx}\|^2d\tau}.
\end{align}
%where $\omega$ defined in \eqref{omega0}%as
%\begin{align}
%\omega(x,t)=(1+t)^{-\frac{1}{2}}exp\left\{-\frac{\alpha x^2}{1+t}\right\},
%\end{align}
\end{Lem}
\begin{proof}
Defining
$$f=\int_{-\infty}^x\omega^2dy,$$ one has
\begin{align}\label{h2}
\|f(\cdot,t)\|_{L^\infty}\le 2\alpha^{-\frac{1}{2}}(1+t)^{-\frac{1}{2}},\ \ \ \|f_t(\cdot, t)\|_{L^\infty}\le4\alpha^{-\frac{1}{2}}(1+t)^{-\frac{3}{2}}.
\end{align}
Multiplying \eqref{sys-pert00}$_2\times(R\mathring W-P_+\mathring\Phi)\bar vf$ and integrating the resulting equation over $\mathbb R$ leads to 
\begin{align}\label{poin5}
\begin{aligned}
\frac{1}{2}\int_{\mathbb R}(p_+\mathring\Phi-R\mathring W)^2\omega^2%f_x
dx
&=\int_{\mathbb R}\mathring\Psi_t(R\mathring W-p_+\mathring\Phi)\bar vfdx-\int_{\mathbb R}(R\mathring W-p_+\mathring\Phi)^2\frac{\bar v_x}{\bar v}fdx\\
&\quad-\int_{\mathbb R}\mathring Q_1(R\mathring W-p_+\mathring\Phi)\bar vfdx+\int_{\mathbb R}\mu\mathring\Psi_x\left\{(R\mathring W-p_+\mathring\Phi)f\right\}_xdx\\
&\quad-\int_{\mathbb R}\left[\left(\frac{R}{\bar v}\right)_x\mathring W-\left(\frac{p_+}{\bar v}\right)_x\mathring\Phi\right](R\mathring W-p_+\mathring\Phi)\bar vfdx:=\sum_{i=1}^{5}J_i.
\end{aligned}
\end{align}
We estimate $J_i$ as follows:
\begin{align}
\begin{aligned}
J_1&=\left(\int_{\mathbb R}\mathring\Psi(R\mathring W-p_+\mathring\Phi)\bar vfdx\right)_t-\int_{\mathbb R}\mathring\Psi(R\mathring W-p_+\mathring\Phi)_t\bar vfdx-\int_{\mathbb R}\mathring\Psi(R\mathring W-p_+\mathring\Phi)\left(\bar vf\right)_tdx\\
&:=\sum_{i=1}^3J_1^i.
\end{aligned}
\end{align}
By \eqref{sys-pert00}$_{1,3}$, \eqref{Q0}-\eqref{Jx0}, {choosing suitably small $\delta>0$} and $\alpha\in(0,\frac{c_2}{4}]$, we get 
\begin{align}\label{poin1}
\begin{aligned}
J_1^2&=-(\gamma-1)\kappa\int_{\mathbb R}\mathring\Psi\mathring W_{xx}fdx-(\gamma-1)\int_{\mathbb R}\mathring\Psi \mathring Q_2\bar vfdx+\frac{\gamma p_+}{2}\int_{\mathbb R}(\mathring\Psi^2)_x\bar vfdx\\
&\le C\int_{\mathbb R}\left|f\mathring\Psi_x\mathring W_x\right|+\left|\mathring\Psi\mathring W_xf_x\right|dx+\int_{\mathbb R} \Big|\left(\frac{\kappa}{v}-\frac{\kappa}{\bar{v}}\right) \theta_x+\frac{\mu u_x}{v} \mathring\Psi_x-{R}_2-\bar{u}_t \mathring\Psi+\bar{u} {R}_1+\mathring J_2\\
&\qquad-\frac{\kappa}{\bar{v}} \mathring Y_x\Big|\left|\mathring\Psi\bar vf\right|dx-\frac{\gamma p_+}{2}\int_{\mathbb R}\mathring\Psi^2\bar vf_xdx-\frac{\gamma p_+}{2}\int_{\mathbb R}\mathring\Psi^2\bar v_xfdx\\
&\le C\|f\|_{L^\infty}\|\mathring\Psi_x\|\|\mathring W_x\|+C\|\mathring\Psi\|_{L^\infty}\|\mathring W_x\|\|\omega^2\|+C\|f\|_{L^\infty}\|\mathring\Psi\|_{L^\infty}\Big[\|\mathring\Phi_x\|^2+\|\mathring\Psi_x\|^2+\|\mathring W_x\|^2\\
&\qquad+\int_{\mathbb R}(\mathring Y^2+|\bar u|^4)dx+\|\mathring\Phi_x\|\|\theta_x\|+\|\mathring\Psi_x\|\|u_x\|+\|R_2\|_{L^1}+\|\mathring\Psi\|_{L^\infty}\|\bar u_t\|_{L^1}+\|\bar uR_1\|_{L^1}\\
&\qquad+\|\mathring\Psi_x\|\|\mathring\Psi_{xx}\|+\int_{\mathbb R}(|\bar u_{xx}\mathring\Psi|+|\bar u_x\mathring\Psi_x|)dx\Big]-\frac{\gamma p_+}{2}\int_{\mathbb R}\mathring\Psi^2\bar v\omega^2dx+C\delta\int_{\mathbb R}\mathring\Psi^2\omega^2dx\\%+C\bar\delta\frac{\gamma p_+}{2}\int_{\mathbb R}\mathring\Psi^2(1+t)^{-1}e^{-\frac{c_2x^2}{1+t}}dx\\
&\le C\|(\mathring\Phi,\mathring\Psi,\mathring W)_x\|^2+{C\|(\mathring\Psi,\mathring W)_{xx}\|^2}%\blue{C\|(\psib,\mathring\zeta)_{x}\|^2}
+C\delta(1+t)^{-\frac{3}{2}}-c%\frac{\gamma p_+}{2}
\int_{\mathbb R}\mathring\Psi^2%\bar v
\omega^2dx,%+C\delta\int_{\mathbb R}\mathring\Psi^2\omega^2dx.
\end{aligned}
\end{align}
where %$c>0$ and 
we use the fact that
\begin{align}\label{poin9}
\left(\frac{R}{\gamma-1}\mathring W+p_+\mathring\Phi\right)_t=\frac{\kappa}{\bar v}\mathring W_{xx}+\mathring Q_2.
\end{align}
%and $\bar\delta>0$ is suitably small, $\alpha\in(0,\frac{c_2}{4}]$. % and constant $c>0$.% $\bar\delta>0$.
Using a priori assumptions, one has %It holds that
\begin{align}\label{poin2}
J_1^3\le C\|(\bar vf)_t\|_{L^\infty}\|\mathring\Psi\|(\|\mathring W\|+\|\mathring\Phi\|)\le C(1+t)^{-\frac{5}{4}},
\end{align}
where we claim that $\|(\mathring\Phi,\mathring\Psi,\mathring W)\|^2\le Cln(1+t)\le C(1+t)^{\frac{1}{4}}$ by \eqref{apa0}.
Then we have
\begin{align}\label{poin4}
\begin{aligned}
J_1\le &\left(\int_{\mathbb R}\mathring\Psi(R\mathring W-p_+\mathring\Phi)\bar vfdx\right)_t-c\int_{\mathbb R}\mathring\Psi^2\omega^2dx+C\|(\mathring\Phi,\mathring\Psi,\mathring W)_x\|^2\\
&+C\|(\mathring\Psi,\mathring W)_{xx}\|^2+C(1+t)^{-\frac{5}{4}}.
\end{aligned}
\end{align}
From \eqref{cde}, \eqref{omega0} and \eqref{h2}, we deduce that
\begin{align}
J_2+J_5\le C\|f\|_{L^\infty}\int_{\mathbb R}(\mathring\Phi^2+\mathring W^2)\delta(1+t)^{-\frac{1}{2}}e^{-\frac{c_2x^2}{1+t}}dx\le C\delta\int_{\mathbb R}(\mathring\Phi^2+\mathring W^2)\omega^2dx.
\end{align}
Similar to \eqref{poin1}, we have
\begin{align}
\begin{aligned}
J_4&\le C\|f\|_{L^\infty}\|\mathring\Psi_x\|\left(\|\mathring W_x\|+\|\mathring\Phi_x\|\right)+C\|f_x\|_{L^\infty}\|\mathring\Psi_x\|\left(\|\mathring W\|+\|\mathring\Phi\|\right)\\
&\le C(1+t)^{-\frac{1}{2}}\|(\mathring\Phi,\mathring\Psi,\mathring W)_x\|^2+C(1+t)^{-\frac{3}{2}}\|(\mathring\Phi,\mathring W)\|^2\\
&\le C\|(\mathring\Phi,\mathring\Psi,\mathring W)_x\|^2+C(1+t)^{-\frac{5}{4}}.
\end{aligned}
\end{align}
Finally, we obtain by \eqref{Q0} that
\begin{align}\label{poin3}
\begin{aligned}
J_3&\le C\int_{\mathbb R}\left|\left(\frac{\mu}{v}-\frac{\mu}{\bar{v}}\right) u_x+\mathring J_1+\frac{R}{\bar{v}} \mathring Y-{R}_1\right|\left|\left(R\mathring W-p_+\mathring\Phi\right)\bar vf\right|dx\\
&\le C\|(\mathring\Phi,\mathring\Psi,\mathring W)_x\|^2+C\delta\int_{\mathbb R}(\mathring\Phi^2+\mathring\Psi^2+\mathring W)\omega^2dx+C\|\mathring\Psi_{xx}\|^2+C\delta(1+t)^{-\frac{3}{2}}.%J_1+(\frac{\mu}{v}-\frac{\mu}{\bar v})
\end{aligned}
\end{align}
Together with all the estimates \eqref{poin4}-\eqref{poin3}, we conclude from \eqref{poin5} that
\begin{align}\label{poin6}
\begin{aligned}
\frac{1}{2}&\int_{\mathbb R}(p_+\mathring\Phi-R\mathring W)^2\omega^2%f_x
dx+c\int_{\mathbb R}\mathring\Psi^2\omega^2dx\\
&\le \left(\int_{\mathbb R}\mathring\Psi(R\mathring W-p_+\mathring\Phi)\bar vfdx\right)_t+C\delta\int_{\mathbb R}(\mathring\Phi^2+\mathring\Psi^2+\mathring W^2)\omega^2dx\\
&\qquad+C\|(\mathring\Phi,\mathring\Psi,\mathring W)_x\|^2+C\|(\mathring\Psi,\mathring W)_{xx}\|^2+C(1+t)^{-\frac{5}{4}}.
\end{aligned}
\end{align}
Integrating \eqref{poin6} over $(0,t)$, one has
\begin{align}\label{poin7}
\begin{aligned}
%\frac{1}{2}
&\int_0^t\int_{\mathbb R}(p_+\mathring\Phi-R\mathring W)^2\omega^2%f_x
dxd\tau+%c
\int_0^t\int_{\mathbb R}\mathring\Psi^2\omega^2dxd\tau\\
&\le C+C\delta\int_0^t\int_{\mathbb R}(\mathring\Phi^2+\mathring W^2)\omega^2dxd\tau+C\int_0^t\|(\mathring\Phi,\mathring\Psi,\mathring W)_x\|^2d\tau+C\int_0^t\|(\mathring\Psi,\mathring W)_{xx}\|^2d\tau.
\end{aligned}
\end{align}

In order to get \eqref{poin8},  we need to %use Lemma \ref{heat0} to 
deduce another main estimate by using Lemma \ref{heat0}. Taking $h=R\mathring W+(\gamma-1)p_+\mathring\Phi$ in lemma \ref{heat0} and use \eqref{poin9}, we obtain
\begin{align}\label{poin11}
\langle{h_t,hg^2}\rangle_{H^{-1}\times H^1}=(\gamma-1)\int_{\mathbb R}\frac{\kappa}{\bar v}\mathring W_{xx}hg^2dx+(\gamma-1)\int_{\mathbb R} \mathring Q_2hg^2dx.
\end{align}
\begin{comment}
where we use the fact that
\begin{align}
\left(\frac{R}{\gamma-1}\mathring W+p_+\mathring\Phi\right)_t=\frac{\kappa}{\bar v}\mathring W_{xx}+\mathring Q_2.
\end{align}
\end{comment}
Using an integration by part about $x$, we have
\begin{align}\label{poin10}
\begin{aligned}
\int_{\mathbb R}\frac{\kappa}{\bar v}\mathring W_{xx}hg^2dx&=-\int_{\mathbb R}\left(\frac{\kappa}{\bar v}\right)_x\mathring W_{x}hg^2dx-\int_{\mathbb R}\frac{\kappa}{\bar v}\mathring W_{x}h_xg^2dx-2\int_{\mathbb R}\frac{\kappa}{\bar v}\mathring W_{x}hg\omega dx\\
&\le C(\delta+\beta)\int_{\mathbb R}(\mathring\Phi^2+\mathring W^2)\omega^2dx+C_{\beta}\|\mathring W_x\|^2+C\|\mathring\Phi_x\|^2,
\end{aligned}
\end{align}
where $\beta>0$ will be determined later and constant $C_\beta$ is only depending on $\beta$.
Now we turn to estimate the last term of \eqref{poin11}. Recalling the definition of $\mathring Q_2$ in \eqref{Q0} and performing the same methods used as \eqref{poin1}, we obtain by direct calculation
\begin{align}\label{poin12}
\begin{aligned}
\int_{\mathbb R} &\left|\mathring Q_2\right|\left|hg^2\right|dx\le\int_{\mathbb R}\left|\left(\frac{\kappa}{v}-\frac{\kappa}{\bar{v}}\right) \theta_x+\frac{\mu u_x}{v} \mathring\Psi_x-{R}_2-\bar{u}_t \mathring\Psi+\bar{u} {R}_1+\mathring J_2-\frac{\kappa}{\bar{v}} \mathring Y_x\right|\left|hg^2\right|dx\\
&\le C\|g^2\|_{L^\infty}\|\mathring\Phi_{x}\|\|\bar\theta_xh\|+C\|hg^2\|_{L^\infty}\|\mathring\Phi_x\|\|\mathring\zeta_x\|+C\|hg^2\|_{L^\infty}\|\mathring\Psi_x\|\|u_x\|+C \delta(1+t)^{-\frac{3}{2}}\\
&\qquad+C\int_{\mathbb R}(|\mathring W|+|\mathring\Phi|)\delta(1+t)^{-\frac{3}{2}}e^{-\frac{c_2x^2}{1+t}}dx+C\|(\mathring\Phi,\mathring\Psi,\mathring W)_x\|^2+C(\|\mathring\Psi_x\|^2+\|\mathring\Psi_{xx}\|^2)\\
&\le C\delta\int_{\mathbb R}(\mathring\Phi^2+\mathring W^2)\omega^2dx+C\|(\mathring\Phi,\mathring\Psi,\mathring W)_x\|^2+C\|(\mathring\Psi,\mathring W)_{xx}\|^2+C \delta(1+t)^{-\frac{3}{2}}.
\end{aligned}
\end{align}
Substituting \eqref{poin10}-\eqref{poin12} into \eqref{poin11}, choosing any suitably small $\beta>0$, one has
\begin{align}\label{poin13}
\begin{aligned}
\langle{h_t,hg^2}\rangle_{H^{-1}\times H^1}\le &C(\delta+\beta)\int_{\mathbb R}(\mathring\Phi^2+\mathring W^2)\omega^2dx+C_\beta\|(\mathring\Phi,\mathring\Psi,\mathring W)_x\|^2\\
&+C\|(\mathring\Psi,\mathring W)_{xx}\|^2+C \delta(1+t)^{-\frac{3}{2}}.
\end{aligned}
\end{align}
It follows from Lemma \ref{heat0} and \eqref{poin13} that 
\begin{align}\label{poin14}
\begin{aligned}
\int_0^t&\int_{\mathbb R}(R\mathring W+(\gamma-1)p_+\mathring\Phi)^2\omega^2dxd\tau\\
&\le C+C(\delta+\beta)\int_0^t\int_{\mathbb R}(\mathring\Phi^2+\mathring W^2)\omega^2dx+C_\beta\int_0^t\|(\mathring\Phi,\mathring\Psi,\mathring W)_x\|^2d\tau+C\int_0^t\|(\mathring\Psi,\mathring W)_{xx}\|^2d\tau.
\end{aligned}
\end{align}
In fact, adding \eqref{poin7} to \eqref{poin14} and taking first $\beta$ then $\delta$ suitably small thus implies \eqref{poin8}. This completes the proof of Lemma \ref{lem5}.

\end{proof}
\begin{comment}
{By \eqref{Z}, it is obvious that %$\left|\bar{\theta}_1\right|+\left|\bar{\theta}_3\right| \leqslant C \varepsilon_0$
\begin{align}\label{smallconstants}
{\left|\bar{\theta}_1\right|+\left|\bar{\theta}_3\right| \leqslant C\varepsilon_0,}%\bar\theta_i\le\e_0.
\end{align}
 for some constant $C>0$.} 
 \end{comment}
% \blue{For simplicity, we denote $\bar{\delta}:=\varepsilon+\chi+\delta.$}
 
Denoting $\mathring{\mathcal{W}}=(\mathring\Phi, \mathring\Psi, \mathring W)^t$, one has,
\begin{align}\label{equ-W0}
\mathring{\mathcal{W}}_t+A_1 \mathring{\mathcal{W}}_x=A_2 \mathring{\mathcal{W}}_{x x}+\mathring A_3,
\end{align}
where
$$\mathring A_3=\left(0,\mathring Q_1,\frac{\gamma-1}{R}\mathring Q_2\right)^t.$$
Let 
\begin{align}
\mathring B=L\mathring{\mathcal{W}}=(\mathring b_1,\mathring b_2,\mathring b_3)^t,
\end{align}
Multiplying the equations \eqref{equ-W0} by the matrix $L$ yields
\begin{align}\label{mat0}
\mathring B_t+\Lambda\mathring B_x=LA_2R\mathring B_{xx}+2LA_2R_x\mathring B_x+\left[(L_t+\Lambda L_x)R+LA_2R_{xx}\right]\mathring B+L\mathring A_3.
\end{align}
Applying $\px^k$, $k=0,1,2$ to \cref{mat0}, one has
\begin{align}\label{equ-Bk0}
	\begin{aligned}
		\px^k\mathring B_t+\Lambda \px^k\mathring B_x=L A_2 R \px^k\mathring B_{x x}+\mathring{\mathcal{M}}_k,
	\end{aligned}
\end{align}
where
\begin{align}
	\begin{aligned}
	\mathring{\mathcal{M}}_k:=&\sum_{j=1}^{k}\bigg[-\px^j\Lambda\px^{k-j+1}\mathring B+\px^j(LA_2R)\px^{k-j+2}\mathring B\bigg]+2\sum_{i=0}^{k}\px^i\big(L A_2 R_x\big) \px^{k-i+1}\mathring B\\
	&+\sum_{i=0}^{k}\bigg\{\px^i\big[\left(L_t+\Lambda L_x\right) R+L A_2 R_{x x}\big] \px^{k-i}\mathring B+\px^{i}L \px^{k-i}\mathring A_3 \bigg\}\\
	\le&O(1)\sum_{j=1}^kD_{-\frac{j}{2}}%\red{\px^{k-j+1}B}
	{\big(\abs{\px^{k-j+1}\mathring b_1}+\abs{\px^{k-j+1}\mathring b_3}\big)}+\sum_{j=1%0
	}^{k+2}|D_{-\frac{j}{2}}\px^{k-j+2}\mathring B|+\sum_{i=0}^{k}|\px^{i}L \px^{k-i}\mathring A_3|\\
	:=&\sum_{i=1}^{3}|\mathring{\mathcal{M}}_{k}^{(i)}|.
\end{aligned}
\end{align}

\subsection{Energy estimates}\label{Ees}
%Since the local existence of the solution is well known, the details are omitted. 
To prove the %global existence of solution in
 Theorem \ref{mt0}, it is only needed to get the following a priori estimates. %is
\begin{Prop} 
	Under the same assumptions as \cref{mt0}, let $(\mathring\Phi,\mathring\Psi,\mathring W,\mathring\zeta)$ be the solutions of \cref{sys-pert00} in $[0,t]$ satisfying \cref{apa0}. Then it holds that
	\begin{align}
	\begin{aligned}
\|(\mathring\Phi,& \mathring\Psi, \mathring W)\|_{L^{\infty}}^2+(1+t)%^{\frac{1}{2}}
\ln^{-1}(t+2)\|(\phib, \psib, \mathring\zeta)\|^2+(1+t)^2\ln^{-1}(t+2)%{\frac{3}{2}}
\|(\phib_x,\phib_{xx}, \psib_x, \mathring\zeta_x)\|^2\\
&\leq O(\varepsilon_0^2+\delta_0^{\frac{1}{2}}).
\end{aligned}
\end{align}
\end{Prop}
\begin{proof}

Multiplying \eqref{equ-Bk0} by $\mathring{\bar B}^{(k)}=(v_1^n\p_x^k\mathring b_1,\p_x^k\mathring b_2, v_1^{-n}\p_x^k\mathring b_3)$ and then integrating the resulting equation on $\mathbb{R}$, we get
\begin{align}\label{mat1}
\begin{aligned}
%&\int_{\mathbb R}\left(\frac{v_1^n}{2}\mathring b_1^2+\frac{1}{2}\mathring b_2^2+\frac{v_1^{-n}}{2}\mathring b_3^2\right)_t+\mathring{\bar B}_xA_4\mathring B_xdx+\int_{\mathbb R}a_1\mathring b_1^2+a_3\mathring b_3^2dx=\int_{\mathbb R}-\mathring{\bar B}A_{4x}\mathring B_x\\
%&+\left[\left(\frac{v_1^n}{2}\right)_t\mathring b_1^2+\left(\frac{v_1^{-n}}{2}\right)_t\mathring b_3^2\right]+2\mathring{\bar B}LA_2R_x\mathring B_x+\mathring{\bar B}\left[L_tR+LA_2R_{xx}\right]\mathring B+\mathring{\bar B}\Lambda L_xR\mathring B+\mathring {\bar B}L\mathring A_3dx.
&\int_{\mathbb R}\left(\frac{v_1^n}{2}\left|\p_x^k\mathring b_1\right|^2+\frac{1}{2}\left|\p_x^k\mathring b_2\right|^2+\frac{v_1^{-n}}{2}\left|\p_x^k\mathring b_3\right|^2\right)_t+\mathring{\bar B}^{(k)}_xA_4\p_x^{k+1}\mathring B_xdx+\int_{\mathbb R}a_1\left|\p_x^k\mathring b_1\right|^2+a_3\left|\p_x^k\mathring b_3\right|^2dx\\
&=\int_{\mathbb R}-\mathring{\bar B}^{(k)}A_{4x}\p_x^k\mathring B_x+\left[\left(\frac{v_1^n}{2}\right)_t\left|\p_x^k\mathring b_1\right|^2+\left(\frac{v_1^{-n}}{2}\right)_t\left|\p_x^k\mathring b_3\right|^2\right]+\mathring{\bar B}^{(k)}\mathring M_kdx.%+2\mathring{\bar B}LA_2R_x\mathring B_x+\mathring{\bar B}\left[L_tR+LA_2R_{xx}\right]\mathring B+\mathring{\bar B}\Lambda L_xR\mathring B+\mathring {\bar B}L\mathring A_3dx.
\end{aligned}
\end{align}
Similar to \eqref{EKG}-\eqref{Ki}, we denote
\begin{align}\label{EKG0}
\begin{aligned}
	&\mathring{\tilde{E}}_k:=\int_{\R}\frac{v_1^n}{2} |\px^k\mathring b_1|^2+\frac{1}{2}| \px^k\mathring b_2|^2+\frac{v_1^{-n}}{2} |\px^k\mathring b_3|^2dx,\quad \mathring{\tilde{K}}_k:=\int_{\R}\px^{k+1}\mathring B^tA_4\px^{k+1}\mathring Bdx,\\
	%\quad
	& \mathring G_k:=\int_{\R} a_1|\px^k\mathring b_1|^2+ a_3|\px^k\mathring b_3|^2dx.
	\end{aligned}
\end{align}
and %Note that $\norm{\Phi_x}_{H^2}^2$ is not include in $\tilde{K}_i$, $i=0,1,2$, we further denote
\begin{align}\label{Ki0}
	\mathring K_i:=\int_{\R}\abs{\px^{i+1}\mathring b_1}^2+\abs{\px^{i+1}\mathring b_2}^2+\abs{\px^{i+1}\mathring b_3}^2dx=O(1)\int_{\R}\abs{\px^{i+1}\mathring\Phi}^2+\abs{\px^{i+1}\mathring\Psi}^2+\abs{\px^{i+1}\mathring W}^2dx,
\end{align} 
for $i=0,1,2$.

Since the estimates of other terms are similar to that in the non-zero initial mass case, we only focus on estimating the terms containing $\mathring A_3$. For $k=0$, 
\begin{align}
\int_{\mathbb R}\mathring{\bar B}^{(0)}L\mathring A_3dx\le C\int_{\mathbb R}\left|\left(\frac{v_1^{-n}}{\lambda_3}\mathring b_3+%-
\frac{v_1^{n}}{\lambda_3}\mathring b_1\right)\mathring Q_1\right|+\left|\left(v_1^n\mathring b_1+\mathring b_2+v_1^{-n}\mathring b_3\right)\mathring Q_2\right|dx:=\mathring I^1_1+\mathring I^1_2.%O(1)\int_{\mathbb R}\left|\mathring{\bar B}^{(0)}\right|\left(
\end{align}
By \eqref{Q0} and \eqref{vr}, it yields
\begin{align}
\begin{aligned}
\mathring I^1_1&\le C\int_{\mathbb R}\left|\frac{v_1^{-n}}{\lambda_3}\mathring b_3+%-
\frac{v_1^{n}}{\lambda_3}\mathring b_1\right| \left|\left(\frac{\mu}{v}-\frac{\mu}{\bar{v}}\right) u_x+\mathring J_1+\frac{R}{\bar{v}} \mathring Y\right|+\left|\frac{v_1^{-n}}{\lambda_3}\mathring b_3+%-
\frac{v_1^{n}}{\lambda_3}\mathring b_1\right|\left|{R}_1\right|dx\\
&\le C\bar\delta\left(\left\|\mathring\Phi_x\right\|+\mathring{\tilde{K}}_0+\mathring{\tilde{K}}_1\right)+C\bar\delta(1+t)^{-1}\mathring{\tilde E}_0+C\bar\delta \mathring G_0+C\bar\delta(1+t)^{-1},
\end{aligned}
\end{align}
and
\begin{align}
\begin{aligned}
\mathring I^1_2&\le C\int_{\mathbb R}\left|%\frac{v_1^{-n}}{\lambda_3}\mathring b_3-\frac{v_1^{n}}{\lambda_3}\mathring b_1
v_1^n\mathring b_1+\mathring b_2+v_1^{-n}\mathring b_3\right| \left|%\left(\frac{\mu}{v}-\frac{\mu}{\bar{v}}\right) u_x+\mathring J_1+\frac{R}{\bar{v}} \mathring Y
\left(\frac{\kappa}{v}-\frac{\kappa}{\bar{v}}\right) \theta_x+\frac{\mu u_x}{v} \mathring\Psi_x-\bar{u}_t \mathring\Psi+\mathring J_2-\frac{\kappa}{\bar{v}} \mathring Y_x-{R}_2+\bar{u} {R}_1\right|dx\\%+\left|%\frac{v_1^{-n}}{\lambda_3}\mathring b_3-\frac{v_1^{n}}{\lambda_3}\mathring b_1
%v_1^n\mathring b_1+\mathring b_2+v_1^{-n}\mathring b_3\right|\left|{R}_2+\bar{u} {R}_1\right|dx\\
&\le C\bar\delta\left(\left\|\mathring\Phi_x\right\|+\mathring{\tilde{K}}_0+\mathring{\tilde{K}}_1\right)+C\bar\delta(1+t)^{-1}\mathring{\tilde E}_0%+C\bar\delta \mathring G_0
+C\bar\delta(1+t)^{-\frac{3}{2}}.
\end{aligned}
\end{align}
Then we have 
\begin{align}\label{b01}
\int_{\mathbb R}\mathring{\bar B}^{(0)}L\mathring A_3dx\le C\bar\delta\left(\left\|\mathring\Phi_x\right\|+\mathring{\tilde{K}}_0+\mathring{\tilde{K}}_1\right)+C\bar\delta(1+t)^{-1}\mathring{\tilde E}_0+C\bar\delta \mathring G_0
+C\bar\delta(1+t)^{-1}.
\end{align}
For $k=1$, there exists
\begin{align}
\begin{aligned}
\int_{\mathbb R}\mathring{\bar B}^{(1)}\left(\px L\mathring A_3+L\px\mathring A_3\right)dx\le C\int_{\mathbb R}&\left|\left(%\frac{v_1^{-n}}{\lambda_3}
v_1^{-n}\px\mathring b_3+%\frac{v_1^{n}}{\lambda_3}
v_1^{n}\px\mathring b_1\right)\bar v_x\mathring Q_1\right|+\left|\left(%\frac{v_1^{-n}}{\lambda_3}
v_1^{-n}\px\mathring b_3+%\frac{v_1^{n}}{\lambda_3}
v_1^{n}\px\mathring b_1\right)%\bar v_x
\mathring Q_{1x}\right|\\
&+\left|\left(v_1^n\px\mathring b_1+\px\mathring b_2+v_1^{-n}\px\mathring b_3\right)\mathring Q_{2x}\right|dx:=\mathring I^2_1+\mathring I^2_2+\mathring I^2_3.%O(1)\int_{\mathbb R}\left|\mathring{\bar B}^{(0)}\right|\left(
\end{aligned}
\end{align}
Then we have %apply the G-N inequality to estimate
\begin{align}
\begin{aligned}
\mathring I_1^2&\le C\int_{\mathbb R}D_{-\frac{1}{2}}\left|v_1^{-n}\px\mathring b_3+v_1^{n}\px\mathring b_1\right|\left|\mathring Q_1\right|dx\\
&\le C\bar\delta%(
\mathring K_1%+\mathring K_2)
+C\bar\delta(1+t)^{-1}(\mathring{\tilde E}_1+(1+t)^{-1}\mathring{\tilde E}_0)+C\bar\delta\mathring G_1+C\bar\delta(1+t)^{-2},
\end{aligned}
\end{align}
\begin{align}
\begin{aligned}
\mathring I_2^2%&\le C\int_{\mathbb R}D_{-\frac{1}{2}}\left|v_1^{-n}\px\mathring b_3+v_1^{n}\px\mathring b_1\right||\mathring Q_1|dx\\
\le C\bar\delta(\mathring K_1+\mathring K_2)+C\bar\delta(1+t)^{-1}(\mathring{\tilde E}_1+(1+t)^{-1}\mathring{\tilde E}_0)+C\bar\delta\mathring G_1+C\bar\delta(1+t)^{-2},
\end{aligned}
\end{align}
\begin{align}
\begin{aligned}
\mathring I_3^2%&\le C\int_{\mathbb R}D_{-\frac{1}{2}}\left|v_1^{-n}\px\mathring b_3+v_1^{n}\px\mathring b_1\right||\mathring Q_1|dx\\
&\le C\bar\delta(\mathring K_1+\mathring K_2)+C\bar\delta(1+t)^{-1}(\mathring{\tilde E}_1+(1+t)^{-1}\mathring{\tilde E}_0)%+C\bar\delta\mathring G_1
+C\bar\delta(1+t)^{-\frac{5}{2}},
\end{aligned}
\end{align}
where 
\begin{align}
\int_{\mathbb R}\left|v_1^{-n}\px\mathring b_3+v_1^{n}\px\mathring b_1\right|\left(D_{-\frac{1}{2}}R_1+R_{1x}\right)dx\le C\bar\delta\mathring G_1+C\bar\delta(1+t)^{-2}.
\end{align}
Then one has %we have
\begin{align}\label{b02}
\begin{aligned}
\int_{\mathbb R}&\mathring{\bar B}^{(1)}\left(\px L\mathring A_3+L\px\mathring A_3\right)dx\\
&\le C\bar\delta(\mathring K_1+\mathring K_2)+C\bar\delta(1+t)^{-1}(\mathring{\tilde E}_1+(1+t)^{-1}\mathring{\tilde E}_0)+C\bar\delta\mathring G_1+C\bar\delta(1+t)^{-2}.
\end{aligned}
\end{align}
Similarly, for $k=2$, we can calculate the following estimates:
\begin{align}\label{b03}
\begin{aligned}
\int_{\mathbb R}&\px\mathring{\bar B}^{(2)}\left(\px L\mathring A_3+L\px\mathring A_3\right)dx\\
&\le C\int_{\mathbb R}\left|\left(%\frac{v_1^{-n}}{\lambda_3}
v_1^{-n}\p_x^3\mathring b_3+%\frac{v_1^{n}}{\lambda_3}
v_1^{n}\p_x^3\mathring b_1\right)\bar v_x\mathring Q_1\right|+\left|\frac{1}{\lambda_3}\left(%\frac{v_1^{-n}}{\lambda_3}
v_1^{-n}\p_x^3\mathring b_3+%\frac{v_1^{n}}{\lambda_3}
v_1^{n}\p_x^3\mathring b_1\right)%\bar v_x
\mathring Q_{1x}\right|\\
&\qquad+\left|\left(v_1^n\p_x^3\mathring b_1+\p_x^2\mathring b_2+v_1^{-n}\p_x^2\mathring b_3\right)\mathring Q_{2x}\right|dx\\
&\le C\bar\delta\mathring K_2+C\bar\delta\sum_{i=0}^2(1+t)^{-(3-i)}\mathring{\tilde E}_i+C\bar\delta^{\frac{1}{2}}\mathring G_2+C\bar\delta^{\frac{1}{2}}(1+t)^{-3},
\end{aligned}
\end{align}
where 
\begin{align}
\begin{aligned}
\int_{\mathbb R}&\left|\left(v_1^{-n}\p_x^3\mathring b_3+v_1^{n}\p_x^3\mathring b_1\right)\bar v_xR_1\right|+\left|\frac{1}{\lambda_3}\left(v_1^{-n}\p_x^3\mathring b_3+v_1^{n}\p_x^3\mathring b_1\right)R_{1x}\right|dx\\
&\le C\int_{\mathbb R}D_{-2}\left(v_1^{-n}\p_x^2\mathring b_3+v_1^{n}\p_x^2\mathring b_1\right)+\bar\delta^{-\frac{1}{2}}D_{-2}\left(v_1^{-(n+1)}\p_x^2\mathring b_3+v_1^{n-1}\p_x^2\mathring b_1\right)dx\\
&\le C\bar\delta^{\frac{1}{2}}\mathring G_2+C\bar\delta^{\frac{1}{2}}(1+t)^{-3}.
\end{aligned}
\end{align}
\begin{comment}
\begin{align}
\begin{aligned}
 \int_{\mathbb R}\mathring{\bar B}^{(2)}\px\left(\px L\mathring A_3+L\px\mathring A_3\right)dx&\le C\int_{\mathbb R}\left|\left(%\frac{v_1^{-n}}{\lambda_3}
v_1^{-n}\p_x^2\mathring b_3+%\frac{v_1^{n}}{\lambda_3}
v_1^{n}\p_x^2\mathring b_1\right)\left[(\bar v^2_x+\bar v_{xx})\mathring Q_1+\bar v_x\mathring Q_{1x}+\mathring Q_{1xx}\right]\right|\\%\right|+\left|\left(%\frac{v_1^{-n}}{\lambda_3}
%v_1^{-n}\p_x^2\mathring b_3+%\frac{v_1^{n}}{\lambda_3}
%v_1^{n}\p_x^2\mathring b_1\right)\bar v_x
%\mathring Q_{1x}\right|\\
&\qquad\quad+\left|\left(v_1^n\p_x^2\mathring b_1+\p_x^2\mathring b_2+v_1^{-n}\p_x^2\mathring b_3\right)\mathring Q_{2xx}\right|dx\\
&\le C
\end{aligned}
 \end{align}
\end{comment}

Collecting \eqref{b01}, \eqref{b02} and \eqref{b03}, one has 
\begin{align}\label{b101}
\begin{aligned}
	\frac{d}{dt}\big(\sum_{i=0}^2\mathring{\tilde{E}}_i\big)+\sum_{i=0}^2(\mathring{\tilde{K}}_i+\mathring G_{i})&\leq C\bar{\delta}^{\frac{1}{2}}(1+t)^{-1}\big(\sum_{i=0}^2\mathring{\tilde{E}}_i\big)+C\bar{\delta}^{\frac{1}{2}}(1+t)^{-1}\\
	&\quad+C\bar{\delta}^{\frac{1}{2}}\norm{\mathring\Phi_x}^2{+C\bar\delta^{\frac{1}{2}}\norm{\mathring\Phi_{xx}}^2+C\bar\delta^{\frac{1}{2}}\mathring K_2%\tilde E_3
	},%\label{b101}\\
	\end{aligned}
	\end{align}
	\begin{align}\label{b102}
	\begin{aligned}
	\frac{d}{dt}\big(\sum_{i=1}^2\mathring{\tilde{E}}_i\big)+\sum_{i=1}^2(\mathring{\tilde{K}}_i+\mathring G_{i})&\leq C\bar{\delta}^{\frac{1}{2}}(1+t)^{-1}\big(\sum_{i=1}^2\mathring{\tilde{E}}_i+\mathring G_0\big)+C\bar{\delta}^{\frac{1}{2}}(1+t)^{-2}\\%\nonumber\\
	&\quad+C\bar{\delta}^{\frac{1}{2}}(1+t)^{-2}\mathring{\tilde{E}}_0+C\bar{\delta}^{\frac{1}{2}}\norm{\mathring\Phi_{xx}}^2{+C\bar\delta^{\frac{1}{2}}\mathring K_2}%\tilde E_3}
	,%\label{b102}\\
	\end{aligned}
	\end{align}
	\begin{align}\label{b103}
	\begin{aligned}
	\frac{d}{dt}\mathring{\tilde{E}}_2+\mathring{\tilde{K}}_2+\mathring G_{2}&\leq C\bar{\delta}^{\frac{1}{2}}\sum_{i=0}^%2
	{2}(1+t)^{-(3-i)}\mathring{\tilde{E}}_{i}+C\bar{\delta}^{\frac{1}{2}}(1+t)^{-3}+C\bar\delta^{\frac{1}{2}}\mathring K_2%\blue{+C\bar{\delta}\norm{\Phi_{xxx}}^2}
	%\nonumber\\
	%&\qquad\qquad\qquad\qquad
	\\
	&\quad+C\bar{\delta}^{\frac{1}{2}}\big[(1+t)^{-2}\mathring G_0+(1+t)^{-1}\mathring G_1\big].%\red{+C\bar\delta^{\frac{1}{2}}\tilde E_3}.
	%\label{b103}
	\end{aligned}
\end{align}
%Denote that 

Repeating the estimates in step 2 for system \eqref{sys-Ori0} and \eqref{sys-pert00}, there exist some positive constant $\mathring{\bar C}_{i+1}>1,\mathring{\hat C}_{i+1}>1%c_i
$, ${i=0,1,2,}$ such that 
\begin{align}
\mathring E_i=\mathring{\bar C}_{i+1}\mathring{\tilde E}_i+\mathring{\hat C}_{i+1}\int_{\R}\frac{\mu}{\bar{v}}\abs{\px^{i}\phib}^2-%\blue{\px^{i}\Phi\Psi }
	{\p_x^i\phib\p_x^i\mathring\Psi}%\red{\p_x^i\left(\Phi_x\Psi\right)}
	dx>\int_{\R} \abs{\px^{i}\mathring b_1}^2+\abs{\px^{i}\mathring b_2}^2+\abs{\px^{i}\mathring b_3}^2+\abs{\px^{i}\phib}^2dx.%,\quad {i=0,1,2.}
\end{align}
Then we can obtain
\begin{align}
	&\frac{d}{dt}\big(\sum_{i=0}^2\mathring E_i\big)+\sum_{i=0}^2(\mathring K_i+\mathring G_{i})\leq C\bar{\delta}^{\frac{1}{2}}(1+t)^{-1}\big(\sum_{i=0}^2{\mathring E}_i\big)+C\bar{\delta}(1+t)^{-1},\label{b011}\\
	&\frac{d}{dt}\big(\sum_{i=1}^2{\mathring E}_i\big)+\sum_{i=1}^2(\mathring K_i+\mathring G_{i})\leq C\bar{\delta}^{\frac{1}{2}}(1+t)^{-1}\big(\sum_{i=1}^2{\mathring E}_i+\mathring G_0\big)+C\bar{\delta}(1+t)^{-2}+C\bar{\delta}^{\frac{1}{2}}(1+t)^{-2}{E}_0,\label{b022}\\
	&\frac{d}{dt}{\mathring E}_2+{\mathring K}_2+G_{2}\leq C\bar{\delta}^{\frac{1}{2}}\sum_{i=0}^2(1+t)^{-(3-i)}{\mathring E}_{i}+C\bar{\delta}(1+t)^{-3}+C\bar{\delta}^{\frac{1}{2}}\big[(1+t)^{-2}\mathring G_0+(1+t)^{-1}\mathring G_1\big].\label{b033}
\end{align}

{\bf The proof of Theorem \ref{mt0}.} 
Integrating \eqref{b011} on $[0,t]$ and using Lemma \ref{lem5}, we get
\begin{align}
\sum_{i=0}^2\mathring E_i +\int_0^t\sum_{i=0}^2(\mathring K_i+\mathring G_{i})d\tau\le C(\e^2+\bar\delta^{\frac{1}{2}})\ln(2+t).
\end{align}
Multiplying \eqref{b022} by $(1+t)$ and Integrating the resulting inequality on $[0,t]$, one has
\begin{align}\label{b111}
\sum_{i=1}^2{\mathring E}_i\le C(\e^2+\bar\delta^{\frac{1}{2}})(1+t)^{-1}
\ln(2+t),\ \ \int_0^t(1+\tau)\sum_{i=1}^2(\mathring K_i+\mathring G_{i})d\tau\le C(\e^2+\bar\delta^{\frac{1}{2}})%(1+t)^{-1}
\ln(2+t).
\end{align}
Multiplying \eqref{b033} by $(1+t)^2$ and Integrating the resulting inequality on $[0,t]$, one has
\begin{align}
(1+t)^2{\mathring E}_2+\int_0^t(1+\tau)^2({\mathring K}_2+G_{2})d\tau
\le C(\e^2+\bar\delta^{\frac{1}{2}})%(1+t)^{-1}
\ln(2+t),
\end{align}
it yields
\begin{align}\label{b110}
{\mathring E}_2\le C(\e^2+\bar\delta^{\frac{1}{2}})(1+t)^{-2}\ln(2+t).
\end{align}
Then by \eqref{WZEta0}, \eqref{b111} and \eqref{b110}, it holds that
\begin{align}
\|(\phib,\psib,\mathring\zeta)\|_{L^\infty}\le C(\e+\bar\delta^{\frac{1}{4}})(1+t)^{-\frac{3}{4}}\ln^{\frac{1}{2}}(2+t).%\big)^{\frac{1}{2}}.
\end{align}
Hence, we get the convergence rate of \eqref{L2decay0}-\eqref{decay0} and close the a priori estimates \eqref{apa0}. 
The proof of Theorem \ref{mt0} is completed.
\end{proof}

\smallskip
\noindent {\bf Acknowledgements.} Lingda Xu is supported by the Research Centre for Nonlinear Analysis at The Hong Kong Polytechnic University.

\end{document}